\documentclass[a4paper, 12pt]{amsart}
\usepackage[utf8]{inputenc}
\usepackage[english]{babel}
\usepackage[T1]{fontenc}
\usepackage{amsmath, amssymb, amsfonts}
\usepackage[all]{xy}
\usepackage{enumerate}
\usepackage{graphicx}
\usepackage{geometry}
\usepackage{rotate}
\usepackage{MnSymbol}
\usepackage{tikz}
\usepackage{comment}
\usepackage{hyperref}

\title{Fusion (semi)rings arising from quantum groups}
\author{Amaury Freslon}
\address{Univ. Paris Diderot, Sorbonne Paris Cité, IMJ-PRG, UMR 7586 CNRS, Sorbonne Universit\'es, UPMC Univ. Paris 06, F-75013, Paris, France}
\email{amaury.freslon@imj-prg.fr}
\date{\today}

\theoremstyle{plain}
\newtheorem{thm}{Theorem}[section]
\newtheorem{prop}[thm]{Proposition}
\newtheorem{cor}[thm]{Corollary}
\newtheorem{lem}[thm]{Lemma}

\theoremstyle{definition}
\newtheorem{de}[thm]{Definition}

\theoremstyle{remark}
\newtheorem{rem}[thm]{Remark}

\DeclareMathOperator{\Aut}{Aut}

\DeclareMathOperator{\Irr}{Irr}
\DeclareMathOperator{\Hom}{Hom}

\DeclareMathOperator{\Pol}{Pol}
\DeclareMathOperator{\Proj}{Proj_{\CC}}

\DeclareMathOperator{\Rep}{Rep}
\DeclareMathOperator{\rl}{rl}

\DeclareMathOperator{\Tr}{Tr}
\DeclareMathOperator{\RR}{\mathcal{R}^{+}(\CC)}

\newcommand{\B}{\mathcal{B}}
\newcommand{\C}{\mathbb{C}}
\newcommand{\CC}{\mathcal{C}^{\circ, \bullet}}

\newcommand{\G}{\mathbb{G}}

\newcommand{\N}{\mathbb{N}}
\newcommand{\NCc}{NC^{\circ, \bullet}}
\newcommand{\Pc}{P^{\circ, \bullet}}

\newcommand{\U}{\mathcal{U}^{\circ, \bullet}}
\newcommand{\Z}{\mathbb{Z}}

\newcommand{\ii}{\imath}

\newcommand{\idpart}{|}
\newcommand{\paarpart}{\sqcap}
\newcommand{\baarpartbaustein}{\rotatebox{180}{$\sqcap$}}
\newcommand{\baarpart}{
\mathrel{\vcenter{\offinterlineskip \hbox{$\baarpartbaustein$}}}}
\newcommand{\vierpartrot}{
\mathrel{\vcenter{\offinterlineskip
\hbox{$\baarpart$} \vskip -.1ex \hbox{\hskip .5ex $\shortmid$} \vskip -.1ex \hbox{$\paarpart$}}}}

\begin{document}

\begin{abstract}
We study the fusion semirings arising from easy quantum groups. We classify all the possible \emph{free} ones, answering a question of T. Banica and R. Vergnioux : these are exactly the fusion rings of quantum groups without any nontrivial one-dimensional representation. We then classify the possible groups of one-dimensional representations for general easy quantum groups associated to noncrossing partitions. As an application, we give a unified proof of the Haagerup property for a broad class of easy quantum groups, recovering as special cases previous results by M. Brannan and F. Lemeux. We end with some considerations on the description of the full fusion ring in the general case.
\end{abstract}

\maketitle

\section{Introduction}

Let $G$ be a compact group and consider the set $\Irr(G)$ of equivalence classes of irreducible representations of $G$. Endowing it with the direct sum and tensor product turns $\N[\Irr(G)]$ into a \emph{fusion semiring} $R^{+}(G)$ which carries important properties of the group $G$. Note that $R^{+}(G)$ is simply the fusion semiring of the monoidal category $\Rep(G)$ of finite-dimensional representations of $G$, hence its associated Grothendieck group $R(G)$ can be identified with the first algebraic K-theory group $K_{0}(\Rep(G))$. The monoidal structure of $\Rep(G)$ turns the latter group into a ring which is particularly relevant to the study of KK-theory since it is known to be isomorphic to $KK^{G}(\C, \C)$.

On the opposite side, let $\Gamma$ be a discrete group and consider the category of finite-dimensional corepresentations of the maximal C*-algebra $C^{*}_{\text{max}}(\Gamma)$. Then, irreducible corepresentations have dimension $1$ and are in one-to-one correspondance with elements of $\Gamma$, the tensor product being given by the group law. Thus, the associated fusion semiring is isomorphic to the group semiring $\N[\Gamma]$ and its Grothendieck group is $\Z[\Gamma]$.

The two objects mentioned above can be gathered into a single picture using the theory of \emph{compact quantum groups} of S.L. Woronowicz (see for example \cite{woronowicz1995compact}). To any compact quantum group $\G$, one can associate a fusion semiring $R^{+}(\G)$ which should be thought of as both the representation semiring of $\G$ and the group semiring of the discrete quantum dual $\widehat{\G}$. It is therefore a central object for the study of these quantum groups. We refer the reader to \cite{banica1999fusion} for a broad overview on the problems linked to fusion semirings and their connections with other subjects.

Among compact quantum groups is a very important class defined by T. Banica and R. Speicher in \cite{banica2009liberation} under the name of \emph{easy quantum groups}. The definition roughly proceeds as follows (see Section \ref{sec:preliminaries} for a rigorous definition) :
\begin{enumerate}
\item Pick up a bunch of \emph{partitions} of sets of integers with some compatibility conditions between them.
\item Associate to each partition a linear map between some finite-dimensional vector spaces. The aforementioned compatibility conditions ensuring that we can compose, make tensor products or take adjoints of these maps.
\item Because of the compatibility conditions, there is a unique smallest concrete complete monoidal C*-category (see \cite{woronowicz1988tannaka} for the definition) such that the spaces of morphisms are spanned by linear maps associated to the partitions.
\item To this category is associated a unique compact quantum group by virtue of S.L. Woronowicz' Tannaka-Krein duality theorem, which is called the \emph{easy quantum group associated to the set of partitions}.
\end{enumerate}
Examples of easy quantum groups are S. Wang's free quantum groups $S_{N}^{+}$, $O_{N}^{+}$ and $U_{N}^{+}$ introduced in \cite{wang1995free} and \cite{wang1998quantum}.

As we see in the above description, the object to which we have the more direct access is the representation category of $\G$ (or rather a "generating" part of it). It is therefore natural to look for a purely combinatorial description of the fusion ring of $\G$ in terms of the initial set of partitions. This is what we endeavoured together with M. Weber in \cite{freslon2013representation}. We gave a general description of the fusion rules (hence of the product in the fusion semiring) for all easy quantum groups. However, the general picture was made quite complicated by the presence of \emph{crossing partitions} inducing degeneracies in the constructions. When such phenomena cannot occur, i.e. when considering only \emph{noncrossing partitions}, one can hope for a tractable description of the fusion ring. Some ideas in this direction have been mentioned in the last section of \cite{freslon2013representation} and are the starting point of the present work.

We will be concerned with the notion of \emph{free fusion semiring} in the sense of \cite{banica2009fusion}. More precisely, consider a set $S$ together with an involution $x\mapsto \overline{x}$ and a \emph{fusion operation} $(x, y)\mapsto x\ast y$ which may take $\emptyset$ as a value. Then, the free monoid $F(S)$ on $S$ can be endowed with a similar structure in the following way : if $w = w_{1}\dots w_{n}$ and $w' = w'_{1}\dots w'_{n'}$ are words in $F(S)$, then
\begin{equation*}
\begin{array}{ccc}
\overline{w} & = & \overline{w}_{n}\dots \overline{w}_{1} \\
w\ast w' & = & w_{1}\dots (w_{n}\ast w'_{1})\dots w'_{n'}
\end{array}
\end{equation*}
By convention, $w\ast w' = 0$ if $w_{n}\ast w'_{1} = \emptyset$ or if one of the two words is empty.

\begin{de}\label{de:fusionsemiring}
The \emph{fusion semiring} of $S$ is the abelian semigroup $R^{+}(S) = \N[F(S)]$ endowed with the 
"tensor product" : 
\begin{equation*}
w\otimes w' = \sum_{w = az, w' = \overline{z}b} ab + a\ast b.
\end{equation*}
A semiring $R^{+}$ is said to be \emph{free} if there exists a set $S$, together with an involution and a fusion operation, such that $R^{+}\simeq R^{+}(S)$.
\end{de}

This is very far from the case of compact groups, since the latter always have commutative fusion semirings. It is nevertheless quite close to the case of free groups. In fact, free fusion semirings arose from the following observations :
\begin{enumerate}
\item Several natural classes of "free" quantum groups appear to have free fusion semirings.
\item The structure of free fusion semirings is well-suited to the generalization of "geometric" techniques used on free groups, for example related to Powers' property. 
\end{enumerate}

It was therefore asked in \cite{banica2009fusion} whether there are many easy quantum groups having free fusion semirings. We answer this question in a seemingly disappointing way : the only easy quantum groups having free fusion semirings are those which were already known. More precisely, we prove in Theorem \ref{thm:mainresult} that the elementary obstruction to freeness of having a nontrivial one-dimensional representation is the only one.

We therefore turn our attention to a more general situation. Based on Theorem \ref{thm:mainresult}, we have two parts in the fusion semiring : a "free part" coming from through-partitions of the category and a group of one-dimensional representations. We therefore endeavour to study the latter. It is in fact possible to completely classify the groups which can occur thanks to the free part. This is done in Theorem \ref{thm:1d}. In particular, this group is always cyclic, a fact which was not obvious.

Knowing the fusion rules (i.e. the fusion semiring) of a quantum group is the first step in the study of its algebraic/geometric properties. As an example, we can prove the Haagerup property, a weakening of amenability, for many easy quantum groups using our results. The argument is inspired from \cite{lemeux2013haagerup} and some elementary considerations on a natural length function for easy quantum groups. This recovers several known results but gives a unified and (in some respects) simpler proof.

We would like to emphasize the fact that even though S.L. Woronowicz' theory of compact quantum groups is a nice and convenient way to state our results, proofs in this paper do not make use of any quantum algebraic or operator algebraic technique. In fact, our aim is to understand some specific categories built from partitions. Since the morphism spaces of these categories are spaces of linear operators between finite-dimensional vector spaces, working with simple objects in this situation amounts to considering minimal projections in some matrix algebras. This is done using only combinatorial tools and some basic linear algebra.

To end this introduction, let us outline the organization of the paper. In Section \ref{sec:preliminaries}, we briefly recall some basic facts concerning (noncrossing) partitions and easy quantum groups. We then give in Section \ref{sec:representations} a summary of our work with M. Weber \cite{freslon2013representation}, the results of which will be used all over the present paper. We also solve in the noncrossing case a problem about direct sums of representations which was left open in \cite{freslon2013representation}. Section \ref{sec:free} is the core of the paper. We first explain the "capping technique" used in several proofs and then study the notion of block-stability, leading to Theorem \ref{thm:mainresult}. Building on this, we classify the "free part" of the fusion ring for any category of noncrossing partitions $\CC$ and give examples of all the possible cases. This is continued in Section \ref{sec:one} where we carry out a similar study for one-dimensional representations, ending with a similar classification in Theorem \ref{thm:1d}. Eventually, Section \ref{sec:applications} contains applications of our results. After giving some results concerning a natural length function on easy quantum groups, we are able to prove in one shot the Haagerup property for a large class of quantum groups, including all the previously known easy examples. We end with some partial results concerning the description of the full fusion ring.

The author whiches to thank the referees for their careful reading of the paper and their comments which helped improve the exposition of this work.

\section{Preliminaries}\label{sec:preliminaries}

This section is a reminder of the terminology and notations concerning partitions and easy quantum groups. We refer the reader to \cite{banica2009liberation}, \cite{freslon2013representation} or other papers on the subject for a more substantial introduction and details. Our setting is more general than in most previous works on easy quantum groups because we deal with \emph{colored} partitions. A special case of this framework appeared the work of P. Glockner and W. von Waldenfels \cite{glockner1989relations}, where the algebra of all two-colored pair partitions is introduced and linked to the Schur-Weyl duality for the unitary group $U_{N}$. This example shows the necessity of using colored partitions to go beyond the orthogonal case. Let us also mention that T. Banica and A. Skalski introduced partitions with two colors to study the representation theory of two-parameter quantum groups in \cite{banica2011two} and of some quantum isometry groups in \cite{banica2012quantum}. Their definitions are not the same as ours (in particular concerning vertical concatenation) so that it is not clear whether their works enter our setting or not, even though there are certainly strong connections.

\subsection{Colored partitions}

Easy quantum groups are based on the combinatorics of partitions and in particular noncrossing ones. A \emph{partition} consists in two integers $k$ and $l$ and a partition of the set $\{1, \dots, k+l\}$. We think of it as an upper row of $k$ points, a lower row of $l$ points and some strings connecting these points. If the strings may be put such that they do not cross, the partition will be said to be \emph{noncrossing}. The set of all partitions is denoted by $P$ and the set of all noncrossing partitions is denoted by $NC$.

A maximal set of points which are all connected in a partition is called a \emph{block}. We denote by $b(p)$ the number of blocks of a partition $p$, by $t(p)$ the number of \emph{through-blocks}, i.e. blocks containing both upper and lower points and by $\beta(p) = b(p)-t(p)$ the number of \emph{non-through-blocks}. This work is concerned with a refinement of the notion of partitions : colored partitions.

\begin{de}
A \emph{(two-)colored partition} is a partition with the additional data of a color (black or white) for each point. The set of all colored partitions is denoted by $\Pc$ and the set of noncrossing colored partitions is denoted by $\NCc$.
\end{de}

In the example below, $p_{2}$ has crossings while $p_{1}$ is a noncrossing colored partition.

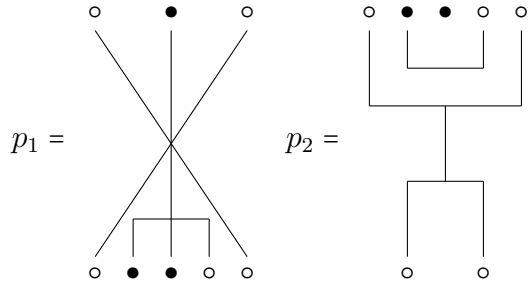
\begin{center}
\begin{tikzpicture}[scale=0.5]
\draw (0,-3) -- (0,3);
\draw (-1,-3) -- (-1,-2);
\draw (1,-3) -- (1,-2);
\draw (-1,-2) -- (1,-2);
\draw (-2,-3) -- (2,3);
\draw (2,-3) -- (-2,3);

\draw (-2,3) node[above]{$\circ$};
\draw (0,3) node[above]{$\bullet$};
\draw (2,3) node[above]{$\circ$};

\draw (-2,-3) node[below]{$\circ$};
\draw (-1,-3) node[below]{$\bullet$};
\draw (0,-3) node[below]{$\bullet$};
\draw (2,-3) node[below]{$\circ$};
\draw (1,-3) node[below]{$\circ$};

\draw (-2.5,0) node[left]{$p_{1} = $};
\end{tikzpicture}
\begin{tikzpicture}[scale=0.5]
\draw (0,-1) -- (0,1);
\draw (-2,1) -- (2,1);
\draw (-2,1) -- (-2,3);
\draw (2,1) -- (2,3);
\draw (-1,2) -- (-1,3);
\draw (1,2) -- (1,3);
\draw (-1,2) -- (1,2);

\draw (-1,-1) -- (1,-1);
\draw (-1,-1) -- (-1,-3);
\draw (1,-1) -- (1,-3);

\draw (-2,3) node[above]{$\circ$};
\draw (-1,3) node[above]{$\bullet$};
\draw (0,3) node[above]{$\bullet$};
\draw (2,3) node[above]{$\circ$};
\draw (1,3) node[above]{$\circ$};

\draw (-1,-3) node[below]{$\circ$};
\draw (1,-3) node[below]{$\circ$};

\draw (-2.5,0) node[left]{$p_{2} = $};
\end{tikzpicture}
\end{center}

From now on, the word "partition" will always mean "two-colored partition". Partitions can be combined using the following \emph{category operations} :

\begin{itemize}
\item  If $p\in \Pc(k, l)$ and $q\in \Pc(k', l')$, then $p\otimes q\in \Pc(k+k', l+l')$ is their \emph{horizontal concatenation}, i.e. the first $k$ of the $k+k'$ upper points are connected by $p$ to the first $l$ of the $l+l'$ lower points, whereas $q$ connects the remaining $k'$ upper points with the remaining $l'$ lower points.
\item If $p\in \Pc(k, l)$ and $q\in \Pc(l, m)$ are such that the coloring of the lower row of $p$ is the same as the coloring of the upper row of $q$, then $qp\in \Pc(k, m)$ is their \emph{vertical concatenation}, i.e. $k$ upper points are connected by $p$ to $l$ middle points and the lines are then continued by $q$ to $m$ lower points. This process may produce loops in the partition. More precisely, consider the set $L$ of elements in $\{1, \dots, l\}$ which are not connected to an upper point of $p$ nor to a lower point of $q$. The lower row of $p$ and the upper row of $q$ both induce partitions of the set $L$. The maximum (with respect to inclusion) of these two partitions is the \emph{loop partition} of $L$, its blocks are called \emph{loops} and their number is denoted by $\rl(q, p)$. To complete the operation, we remove all the loops.
\item If $p\in \Pc(k, l)$, then $p^{*}\in \Pc(l, k)$ is the partition obtained by reflecting $p$ with respect to the horizontal axis (without changing the colors).
\item If $p\in \Pc(k, l)$, then we can shift the very left upper point to the left of the lower row (or the converse) and change its color. We do not change the strings connecting the points in this process. This gives rise to a partition in $\Pc(k-1, l+1)$ (or in $\Pc(k+1, l-1)$), called a \emph{rotated version} of $p$. We can also rotate partitions on the right.
\item Using, the category operations above, one can \emph{reverse} a partition $p$  by rotating all its upper point to the lower row and all its lower points to the upper row. This gives a new partition $\overline{p}$. Note that $\overline{p}$ is in general different from $p^{*}$, because the colors are changed by the rotation.
\end{itemize}

As an example, we give the vertical concatenation of the two partitions $p_{1}$ and $p_{2}$ defined above.
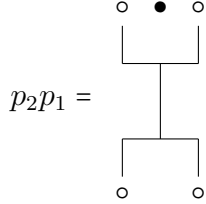
\begin{center}
\begin{tikzpicture}[scale=0.5]
\draw (0,-1) -- (0,1);
\draw (-1,1) -- (1,1);
\draw (-1,1) -- (-1,2);
\draw (1,1) -- (1,2);

\draw (-1,-1) -- (1,-1);
\draw (-1,-1) -- (-1,-2);
\draw (1,-1) -- (1,-2);

\draw (-1,2) node[above]{$\circ$};
\draw (0,2) node[above]{$\bullet$};
\draw (1,2) node[above]{$\circ$};

\draw (-1,-2) node[below]{$\circ$};
\draw (1,-2) node[below]{$\circ$};

\draw (-1.5,0) node[left]{$p_{2}p_{1} = $};
\end{tikzpicture}
\end{center}

There are four ways of coloring the partition $\idpart\in \Pc(1, 1)$. If the two points are white (resp. black), we will call it the \emph{white identity} (resp. \emph{black identity}) partition. Note that these two partitions are rotated versions of each other.

\begin{de}
A \emph{category of partitions} is the data of a set $\CC(k, l)$ of colored partitions  for all integers $k$ and $l$, which is stable under the above category operations and contains the white identity (hence also the black identity).
\end{de}

\begin{rem}\label{rem:orthogonal}
Let $\CC$ be a category of partitions containing the partition $\idpart$ with different colors on the two points. Then, using the category operations we can change the color of any point in any partition of $\CC$. Thus, $\CC$ can be treated as a category of non-colored partitions. In the language of quantum groups, an identity partition with different colors means that the fundamental representation is equivalent to its contragredient, hence the quantum group is in fact a subgroup of the free orthogonal quantum group $O_{N}^{+}$.
\end{rem}

The crucial notion for the study of the representation theory of easy quantum groups is that of \emph{projective partition}.

\begin{de}
A partition $p\in \Pc(k, k)$ is said to be \emph{projective} if it satisfies $pp = p = p^{*}$.
\end{de}

There are actually many of them, according to the following result (see \cite[Prop 2.12]{freslon2013representation}) :

\begin{prop}
A partition $p\in \Pc(k, k)$ is projective if and only if there exists a partition $r\in \Pc(k, k)$ such that $r^{*}r = p$.
\end{prop}

The other ingredient we need is a specific decomposition of partitions called the \emph{through-block decomposition}. Let us call a partition $p$ a \emph{building partition} if it satisfies the following properties :
\begin{enumerate}
\item All lower points of $p$ are colored in white and belong to different blocks.
\item For any lower point $1'\leqslant x'\leqslant l'$ of $p$, there exists at least one upper point which is connected to it and we define $\min_{\text{up}}(x')$ to be the smallest upper point $1\leqslant y\leqslant k$ which is connected to $x'$.
\item For any two lower points $1'\leqslant a' < b'\leqslant l'$ of $p$, we have $\min_{\text{up}}(a') < \min_{\text{up}}(b')$.
\end{enumerate}

We can use building partitions to decompose any partition. Here, we only give the noncrossing version of \cite[Prop 2.9]{freslon2013representation}.

\begin{prop}
Let $p\in \NCc$ be a noncrossing partition. Then, there exists a unique pair $(p_{l}, p_{u})$ of building partitions such that $p = p_{l}^{*}p_{u}$.
\end{prop}

\subsection{Easy quantum groups}

\subsubsection{Partitions and linear maps}

The link between partitions and easy quantum groups lies in the following definition \cite[Def 1.6]{banica2009liberation}. Note that this definition does not involve the coloring of the partitions.

\begin{de}
Let $N$ be an integer and let $(e_{1}, \dots, e_{N})$ be a basis of $\C^{N}$. For any partition $p\in \Pc(k, l)$, we define a linear map
\begin{equation*}
\mathring{T}_{p}:(\C^{N})^{\otimes k} \mapsto (\C^{N})^{\otimes l}
\end{equation*}
by the following formula :
\begin{equation*}
\mathring{T}_{p}(e_{i_{1}} \otimes \dots \otimes e_{i_{k}}) = \sum_{j_{1}, \dots, j_{l} = 1}^{n} \delta_{p}(i, j)e_{j_{1}} \otimes \dots \otimes e_{j_{l}},
\end{equation*}
where $\delta_{p}(i, j) = 1$ if and only if all strings of the partition $p$ connect equal indices of the multi-index $i = (i_{1}, \dots, i_{k})$ in the upper row with equal indices of the multi-index $j = (j_{1}, \dots, j_{l})$ in the lower row. Otherwise, $\delta_{p}(i, j) = 0$.
\end{de}

These maps can be normalized in order to get nicer operator algebraic properties by \cite[Prop 2.18]{freslon2013representation} :

\begin{prop}
Set $T_{p} = N^{\beta(p)/2}\mathring{T}_{p}$ for any partition $p\in \Pc$. Then, $T_{p}$ is a partial isometry. Moreover, $T_{p}$ is a projection if and only if $p$ is a projective partition.
\end{prop}

The interplay between these maps and the category operations are given by the following rules proved in \cite[Prop. 1.9]{banica2009liberation} and \cite[Prop 2.18]{freslon2013representation} :

\begin{itemize}
\item $T_{p}^{*} = T_{p^{*}}$.
\item $T_{p}\otimes T_{q} = T_{p\otimes q}$.
\item $T_{pq} = N^{\gamma(p, q)} T_{p}\circ T_{q}$, where $\gamma(p, q) = (\beta(p) + \beta(q) - \beta(pq))/2 - \rl(p, q)$.
\end{itemize}

It should be stressed that the maps $T_{p}$ are not linearly independent in general. However, restricting to the noncrossing case rules out this problem, see Proposition \ref{prop:linearindependence}.

\subsubsection{Tannaka-Krein duality and quantum groups}

We refer the reader to the original paper \cite{woronowicz1995compact} for a comprehensive treatment of the notion of compact quantum group. Let us consider a compact quantum group $\G$ with a \emph{fundamental representation}, i.e. a finite-dimensional representation $u$ such that any finite-dimensional representation of $\G$ arises as a subrepresentation of some tensor products of $u$ and its contragredient $\overline{u}$. Let us associate to any word $w = w_{1}\dots w_{k}$ in the free monoid $F$ over $\{-1, 1\}$ a representation $u^{\otimes w}$ by setting
\begin{equation*}
u^{\otimes w} = u^{w_{1}}\otimes \dots \otimes u^{w_{k}},
\end{equation*}
where by convention $u^{1} = u$ and $u^{-1} = \overline{u}$. Then, the representation category of $\G$ is completely determined by the intertwiner spaces $\Hom(u^{\otimes w}, u^{\otimes w'})$ for all words $w, w'\in F$. Here, we see the need for two colors in order to treat arbitrary tensor products of $u$ and $\overline{u}$. If we were using only one color, we would have to assume that $u$ is equivalent to $\overline{u}$, i.e. that the quantum groups are orthogonal.

Reciprocally, given a family $(\Hom(w, w'))_{w, w'}$ of finite-dimensional vector spaces with sufficiently nice properties, one can reconstruct the compact quantum group $\G$ using S.L. Woronowicz's Tannaka-Krein theorem \cite[Thm 1.3]{woronowicz1988tannaka}. Let us state this theorem in the particular case which is relevant for us. Note that there is an obvious bijection between colorings and words in $F$ given by
\begin{equation*}
\circ \mapsto 1 \text{ and } \bullet \mapsto -1.
\end{equation*}
If $\CC$ is a category of partitions and if $w, w'\in F$, we will denote by $\CC(w, w')$ the set of partitions $p\in \CC(\vert w\vert, \vert w'\vert)$ such that the upper coloring of $p$ is $w$ and the lower coloring of $p$ is $w'$ (here $\vert w\vert$ denotes the length of the word $w$).

\begin{thm}[Woronowicz]
Let $\CC$ be a category of partitions and let $N$ be an integer. Then, there exists a unique (up to isomorphism) pair $(\G, u)$, where $\G$ is a compact quantum group and $u$ is a fundamental representation of $\G$ such that $\Hom(u^{\otimes w}, u^{\otimes w'})$ is the linear span of the maps $T_{p}$ for $p\in \CC(w, w')$. 
\end{thm}

Such a $\G$ will be called a \emph{(unitary) easy quantum group} or a \emph{partition quantum group}. Let $\U$ be the smallest category of partitions (i.e. the one generated by the white identity partition). The associated quantum group is the \emph{free unitary quantum group} $U_{N}^{+}$ introduced by S. Wang in \cite{wang1995free}. Since inclusion of categories of partitions translates into reversed inclusion of compact quantum groups, we see that any easy quantum group is a quantum subgroup of $U_{N}^{+}$. The other extreme case is the category of all partitions $\Pc$, which yields the symmetric group $S_{N}$. Thus, easy quantum groups form a special class of quantum groups $\G$ in the range
\begin{equation*}
S_{N} \subset \G \subset U_{N}^{+}.
\end{equation*}
Other examples of easy quantum groups include S. Wang's free symmetric quantum group $S_{N}^{+}$ ($\CC = \NCc$) and free orthogonal quantum group $O_{N}^{+}$ ($\CC =$ all partitions with blocks of size $2$). We refer the reader to \cite{wang1995free} and \cite{wang1998quantum} for the definition of these quantum groups and to \cite{banica2009liberation} for proofs of these facts.

As mentioned in Remark \ref{rem:orthogonal}, $\G\subset O_{N}^{+}$ if and only if $\CC$ is stable under any change of coloring. Such \emph{orthogonal easy quantum groups} have been studied in many details and are now completely classified (see for instance \cite{banica2009liberation}, \cite{banica2010classification}, \cite{weber2012classification} and \cite{raum2013full}). The world of unitary easy quantum groups is much more complicated and we will not study these objects in full generality. We will rather restrict ourselves to \emph{noncrossing quantum groups}.

\begin{de}
An easy quantum group $\G$ is said to be \emph{noncrossing} if its associated category of partitions is noncrossing.
\end{de}

In other words, $\G$ is free if and only if $S_{N}^{+} \subset \G$. In that case, the linear independance problem for the maps $T_{p}$ is completely solved (see e.g. \cite[Lem 4.16]{freslon2013representation} for a proof).

\begin{prop}\label{prop:linearindependence}
Let $\CC$ be a category of \emph{noncrossing} partitions and fix an integer $N\geqslant 4$. Then, for any $w, w'\in F$, the maps $(T_{p})_{p\in \CC(w, w')}$ are linearly independent.
\end{prop}

\section{Representations associated to partitions}\label{sec:representations}

\subsection{General structure of the representation theory}

From now on, let us fix a category of \emph{noncrossing} partitions $\CC$, an integer $N\geqslant 4$ (so that we can use Proposition \ref{prop:linearindependence}) and let $(\G, u)$ be the associated easy quantum group. We briefly recall the description of the representations theory of $\G$ given in \cite[Sec 6.2]{freslon2013representation}. For $w\in F$, let $\Proj(w)$ denote the set of projective partitions in $\CC$ with upper (and thus also lower) coloring $w$ and note that the through-block decomposition of a projective partition has the form $p = p_{u}^{*}p_{u}$.

\begin{de}
Two projective partitions $p, q\in \CC$ are said to be \emph{equivalent} if there exists a partition $r\in \CC$ such that
\begin{equation*}
p = r^{*}r \text{ and } q = rr^{*}.
\end{equation*}
In that case, we write $p\sim q$. Note that $p\sim q$ implies that $t(p) = t(q)$. Equivalently, setting $r^{p}_{q} = q_{u}^{*}p_{u}$, we have that $p\sim q$ if and only if $r^{p}_{q} \in \CC$. 
\end{de}

\begin{de}
A projective partition $q\in \CC$ is said to be \emph{dominated} by another projective partition $p\in \CC$ if $pq = qp = q$. This is equivalent to the fact that $T_{p}$ dominates $T_{q}$ as a projection. In that case, we write $q\preceq p$. If moreover $q\neq p$, we write $q\prec p$.
\end{de}

For $p\in \Proj(k)$, we can define a projection $P_{p}\in \Hom(u^{\otimes w}, u^{\otimes w})$ and a representation $u_{p}\subset u^{\otimes w}\in C_{\text{max}}(\G)\otimes \B((\C^{N})^{\otimes \vert w\vert})$ by
\begin{equation*}
P_{p} = T_{p} - \bigvee_{q\prec p} T_{q} \text{ and } u_{p} = (\ii\otimes P_{p})(u^{\otimes k}).
\end{equation*}
According to \cite[Sec 6.2]{freslon2013representation}, the representations $u_{p}$ enjoy the following properties : 
\begin{itemize}
\item $u_{p}$ is non-zero and irreducible for all $p\in \Proj(w)$.
\item Any irreducible representation of $\G$ is unitarily equivalent to $u_{p}$ for some $p$. \item $u_{p}$ is unitarily equivalent to $u_{q}$ if and only if $p\sim q$.
\end{itemize}

\begin{rem}
The above description is rather simple because the category of partitions is assumed to be noncrossing. When crossings are allowed, new problems arise, see \cite[Sec 4]{freslon2013representation}
\end{rem}

Let us now describe the decomposition of the tensor product of $u_{p}$ and $u_{q}$, i.e. the \emph{fusion rules} of $\G$. Subrepresentations of $u_{p}\otimes u_{q}$ are associated to partitions obtained by "mixing" the structure of $p$ and $q$. To explain this, we first need to introduce some specific partitions : we denote by $h_{\square}^{k}$ the projective partition in $\NCc(2k, 2k)$ where the $i$-th point in each row is connected to the $(2k-i+1)$-th point in the same row (i.e. an increasing inclusion of $k$ blocks of size $2$) and all the points are white. If moreover we connect the points $1$, $k$, $1'$ and $k'$, we obtain another projective partition in $\NCc(2k, 2k)$ denoted $h_{\boxvert}^{k}$.

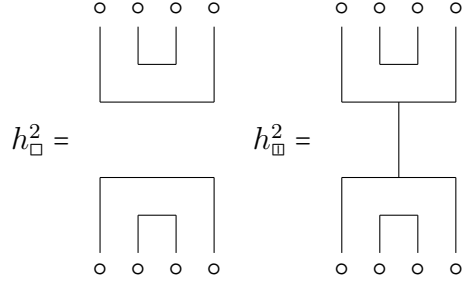
\begin{center}
\begin{tikzpicture}[scale=0.5]
\draw (-0.5,2) -- (0.5,2);
\draw (-1.5,1) -- (1.5,1);
\draw (-0.5,2) -- (-0.5,3);
\draw (0.5,2) -- (0.5,3);
\draw (-1.5,1) -- (-1.5,3);
\draw (1.5,1) -- (1.5,3);

\draw (-0.5,-2) -- (0.5,-2);
\draw (-1.5,-1) -- (1.5,-1);
\draw (-0.5,-2) -- (-0.5,-3);
\draw (0.5,-2) -- (0.5,-3);
\draw (-1.5,-1) -- (-1.5,-3);
\draw (1.5,-1) -- (1.5,-3);

\draw (-1.5,3) node[above]{$\circ$};
\draw (-0.5,3) node[above]{$\circ$};
\draw (0.5,3) node[above]{$\circ$};
\draw (1.5,3) node[above]{$\circ$};

\draw (-1.5,-3) node[below]{$\circ$};
\draw (-0.5,-3) node[below]{$\circ$};
\draw (0.5,-3) node[below]{$\circ$};
\draw (1.5,-3) node[below]{$\circ$};

\draw (-2,0) node[left]{$h_{\square}^{2} = $};
\end{tikzpicture}
\begin{tikzpicture}[scale=0.5]
\draw (-0.5,2) -- (0.5,2);
\draw (-1.5,1) -- (1.5,1);
\draw (-0.5,2) -- (-0.5,3);
\draw (0.5,2) -- (0.5,3);
\draw (-1.5,1) -- (-1.5,3);
\draw (1.5,1) -- (1.5,3);

\draw (0,-1) -- (0,1);

\draw (-0.5,-2) -- (0.5,-2);
\draw (-1.5,-1) -- (1.5,-1);
\draw (-0.5,-2) -- (-0.5,-3);
\draw (0.5,-2) -- (0.5,-3);
\draw (-1.5,-1) -- (-1.5,-3);
\draw (1.5,-1) -- (1.5,-3);

\draw (-1.5,3) node[above]{$\circ$};
\draw (-0.5,3) node[above]{$\circ$};
\draw (0.5,3) node[above]{$\circ$};
\draw (1.5,3) node[above]{$\circ$};

\draw (-1.5,-3) node[below]{$\circ$};
\draw (-0.5,-3) node[below]{$\circ$};
\draw (0.5,-3) node[below]{$\circ$};
\draw (1.5,-3) node[below]{$\circ$};

\draw (-2,0) node[left]{$h_{\boxvert}^{2} = $};
\end{tikzpicture}
\end{center}

From this, we define binary operations on projective partitions (using $\idpart$ to denote the white identity) :
\begin{eqnarray*}
p\square^{k} q & = & (p_{u}^{*}\otimes q_{u}^{*})\left(\idpart^{\otimes t(p)-k}\otimes h_{\square}^{k}\otimes \idpart^{\otimes t(q)-k}\right)(p_{u}\otimes q_{u}) \\
p\boxvert^{k} q & = & (p_{u}^{*}\otimes q_{u}^{*})\left(\idpart^{\otimes t(p)-k}\otimes h_{\boxvert}^{k}\otimes \idpart^{\otimes t(q)-k}\right)(p_{u}\otimes q_{u})
\end{eqnarray*}
for $0\leqslant k\leqslant \min(t(p), t(q))$. We can now state the key result \cite[Thm 6.8]{freslon2013representation} :
\begin{equation*}
u_{p}\otimes u_{q} = u_{p\otimes q}\oplus\sum_{k=1}^{\min(t(p), t(q))} (u_{p\square^{k}q} \oplus u_{p\boxvert^{k}q}),
\end{equation*}
where by convention $u_{r} = 0$ if $r\notin \CC$. We can in fact strengthen this statement by noticing that the projective partitions appearing in the left-hand side have pairwise different number of through-blocks. According to \cite[Prop 4.23]{freslon2013representation}, this implies that they are pairwise orthogonal, hence the sum is a direct sum (see \cite[Rmk 5.8]{freslon2013representation}). We therefore have :
\begin{equation*}
u_{p}\otimes u_{q} = u_{p\otimes q}\oplus\bigoplus_{k=1}^{\min(t(p), t(q))} (u_{p\square^{k}q} \oplus u_{p\boxvert^{k}q}),
\end{equation*}

\begin{rem}
In general, tensor products of such representations are given by the more complicated formula of \cite[Thm 4.27]{freslon2013representation}, where the representations may not be in direct sum.
\end{rem}

\subsection{Direct sum of representations}

The key feature of the family of representations $u_{p}$ is that they in fact yield all irreducible representations up to unitary equivalence. This is a consequence of the decomposition of $u^{\otimes w}$ given in \cite[Thm 6.5]{freslon2013representation} :
\begin{equation}\label{eq:generaldecomposition}
u^{\otimes w} = \sum_{p\in \Proj(w)}u_{p}.
\end{equation}
However, this decomposition is unsatisfying in the sense that it is not proven that the subrepresentations are in direct sum. More precisely, there could be pairwise equivalent projective partitions $p, q_{1}, \dots q_{n}$, all distinct, such that
\begin{equation*}
P_{p} < \bigvee_{i} P_{q_{i}}.
\end{equation*}
This would mean that $u_{p}\subset \sum_{i} u_{q_{i}}$, i.e. $u_{p}$ is redundant in Equation \eqref{eq:generaldecomposition}. Making \cite[Thm 6.5]{freslon2013representation} more precise means characterizing which projective partitions are redundant. Restricting to the noncrossing case, we can solve this problem.

Let us first make some observations. For a projective partition $p\in \Proj(w)$, we denote by $[p]_{w}$ the equivalence class of $p$ in $\Proj(w)$ and by $n_{w}(p)$ the cardinality of that class. Taking the supremum of the projections $P_{q}$ over all $q\in [p]_{w}$ yields a projection $P_{[p]_{w}}$ and an associated representation $u_{[p]_{w}}\subset u^{\otimes w}$. Let $E_{w}(\CC)$ be a system of representatives of the equivalence classes of $\Proj(w)$. By \cite[Prop 4.23]{freslon2013representation} (or rather its straightforward colored generalization), the representations $u_{[p]_{w}}$ and $u_{[q]_{w}}$ are orthogonal as soon as $p$ is not equivalent to $q$ and
\begin{equation}\label{eq:classdecompostion}
u^{\otimes w} = \bigoplus_{p\in E_{w}(\CC)}u_{[p]_{w}}.
\end{equation}
By the orthogonality property, any irreducible subrepresentation of $u_{[p]_{w}}$ must be equivalent to $u_{p}$, hence $u_{[p]_{w}} \sim \nu_{w}(p)u_{p}$ for some integer $0\leqslant \nu_{w}(p)\leqslant n_{w}(p)$. Having redundant projective partitions means that $\nu_{w}(p) < n_{w}(p)$.

\begin{lem}\label{lem:counting}
Let $\CC$ be a category of \emph{noncrossing} partitions and let $w\in F$. Then, $\vert \CC(w, w)\vert = \sum_{p\in E_{w}(\CC)}n_{w}(p)^{2}$, where $\vert D\vert$ denotes the cardinality of a set $D$.
\end{lem}

\begin{proof}
Consider the surjective map
\begin{equation*}
f : \left\{\begin{array}{ccc}
\CC(k, k) & \rightarrow & \Proj(k) \\
r & \mapsto & r^{*}r
\end{array}\right.
\end{equation*}
For any $p\in \Proj(w)$, $f(r) = p$ if and only if there exists $q\in \Proj(w)$ such that $q\sim p$ and $r = r^{p}_{q}$. Hence, $\vert f^{-1}(\{p\})\vert = n_{w}(p)$. Adding up, we get
\begin{equation*}
\vert \CC(w, w)\vert = \sum_{p\in \Proj(w)}n_{w}(p) = \sum_{p\in E_{w}(\CC)}n_{w}(p)^{2}.
\end{equation*}
\end{proof}

\begin{lem}
Let $N\geqslant 4$ be an integer, let $\CC$ be a category of \emph{noncrossing} partitions and let $w\in F$. Then, $\nu_{w}(p) = n_{w}(p)$ for all $p\in \Proj(w)$.
\end{lem}

\begin{proof}
Let us denote by $\Aut(v)$ the set of self-interwiners of a representation $v$. Equation \eqref{eq:classdecompostion} yields
\begin{equation}\label{eq:dimensions}
\dim(\Aut(u^{\otimes w})) = \sum_{p\in E_{w}(\CC)}\dim(\Aut(u_{[p]_{w}})).
\end{equation}
The space $\Aut(u^{\otimes w})$ is known to be generated by the maps $T_{p}$ for $p\in \CC(w, w)$. Moreover, the fact that $\CC$ is noncrossing and that $N\geqslant 4$ imply that the maps $T_{p}$ are linearly independent by Prop \ref{prop:linearindependence}. Thus, the left-hand side of Equation \eqref{eq:dimensions} is equal to $\vert \CC(w, w)\vert$.
On the other hand, using the isomorphism
\begin{equation*}
\Aut(u_{[p]_{w}}) = \Aut(\nu_{w}(p)u_{p}) \simeq M_{\nu_{w}(p)}(\C),
\end{equation*}
we see that the right-hand side is equal to $\sum_{p\in E_{w}(\CC)}\nu_{w}(p)^{2}$. Combining these facts and Lemma \ref{lem:counting}, we have
\begin{equation*}
\sum_{p\in E_{w}(\CC)}n_{w}(p)^{2} = \sum_{p\in E_{w}(\CC)}\nu_{w}(p)^{2}.
\end{equation*}
Since $\nu_{k}(p)\leqslant n_{k}(p)$ for all $p$, we must have equality.
\end{proof}

This result can be restated in the following way : in the noncrossing case, no projective partition is redundant. This yields the following refinement of \cite[Thm 6.5]{freslon2013representation}.

\begin{prop}\label{prop:directsumdecomposition}
Let $N\geqslant 4$ be an integer, let $\CC$ be a category of noncrossing partitions and let $\G$ be the associated easy unitary quantum group. Then, for any $w\in F$, we have
\begin{equation*}
u^{\otimes w} = \bigoplus_{p\in \Proj(w)}u_{p}
\end{equation*}
\end{prop}

\section{Free fusion semirings}\label{sec:free}

\subsection{The capping technique}

This section contains our main results. The proofs are purely combinatorial and heavily rely on the manipulation of partitions using the category operations. In particular, we will use the so-called \emph{capping technique} : given a colored partition $p\in \CC(k, l)$, we may use a one-line partition $b\in \CC(m, 0)$ to produce a new partition $q\in \CC(k, l-m)$
\begin{equation*}
q = (\idpart^{\otimes j_{1}}\otimes b\otimes \idpart^{\otimes j_{2}})p
\end{equation*}
if the colorings fit (and $j_{1} + j_{2} + m = l$). Concretely, this process reduces the partition $p$ by collapsing several neighboring points. Let us express this in another way. Let $p$ be a partition and let $k_{1}, k_{1}+1, \dots, k_{2}$ be a sequence of neighboring points in $p$ such that the one-block partition $b\in \NCc(0, k_{2}-k_{1}+1)$, with the same coloring as the corresponding points of $p$, is in $\CC$. Then, the partition $q$ obtained by removing the points $k_{1}, k_{1}+1, \dots, k_{2}$ and linking all the blocks to which they belong is in $\CC$. If the partition $p$ is symmetric, we may do a \emph{symmetric capping} by capping with the same block on both rows. The following fact is crucial and will be used all over the paper.

\begin{lem}\label{lem:cappingprojective}
Let $\CC$ be a category of noncrossing partitions and let $p\in \CC$ be a projective partition. Then, any projective partition $q$ obtained from $p$ by symmetric capping \emph{and such that $t(q) = t(p)$} is equivalent to $p$.
\end{lem}

\begin{proof}
Let $b$ be the partition used for the capping and let $r$ be the partition obtained by capping only the lower row of $p$. By definition, $rr^{*} = q$ and we have to prove that $r^{*}r = p$. Consider two upper points which are not connected in $p$. If at least one of these points is not connected to a lower point in $p$, then the two points are not connected in $r^{*}r$. Assume now that both points are connected in $p$ to lower points and that they are connected in $r^{*}r$. Then, $b$ connects these two points so that they are already connected in $r$. This implies that $t(r) < t(p)$, which is impossible since $t(r) = t(rr^{*}) = t(q) = t(p)$. We have proved that two points which are not connected in $p$ are not connected in $rr^{*}$. It is clear that if two points are connected in $p$, they are still connected in $rr^{*}$. Since the coloring of $rr^{*}$ is the same as that of $p$, we have $rr^{*} = p$ and the proof is complete.
\end{proof}

Here is an example of equivalence produced by capping two pairs of black and white points :
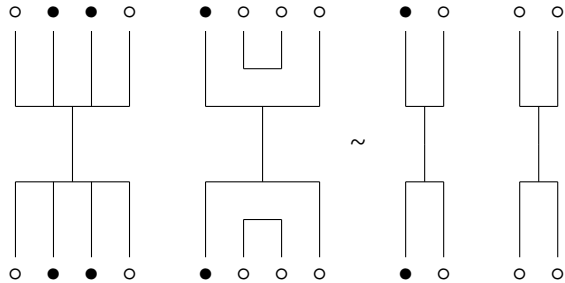
\begin{center}
\begin{tikzpicture}[scale=0.5]
\draw (-2.5,-1) -- (-2.5,1);
\draw (-1,1) -- (-4,1);
\draw (-4,1) -- (-4,3);
\draw (-3,1) -- (-3,3);
\draw (-2,1) -- (-2,3);
\draw (-1,1) -- (-1,3);

\draw (2.5,-1) -- (2.5,1);
\draw (1,1) -- (4,1);
\draw (4,1) -- (4,3);
\draw (3,2) -- (3,3);
\draw (3,2) -- (2,2);
\draw (2,2) -- (2,3);
\draw (1,1) -- (1,3);

\draw (-1,-1) -- (-4,-1);
\draw (-4,-1) -- (-4,-3);
\draw (-3,-1) -- (-3,-3);
\draw (-2,-1) -- (-2,-3);
\draw (-1,-1) -- (-1,-3);

\draw (1,-1) -- (4,-1);
\draw (4,-1) -- (4,-3);
\draw (3,-2) -- (3,-3);
\draw (3,-2) -- (2,-2);
\draw (2,-2) -- (2,-3);
\draw (1,-1) -- (1,-3);

\draw (-4,3) node[above]{$\circ$};
\draw (-3,3) node[above]{$\bullet$};
\draw (-2,3) node[above]{$\bullet$};
\draw (-1,3) node[above]{$\circ$};
\draw (4,3) node[above]{$\circ$};
\draw (3,3) node[above]{$\circ$};
\draw (2,3) node[above]{$\circ$};
\draw (1,3) node[above]{$\bullet$};

\draw (-4,-3) node[below]{$\circ$};
\draw (-3,-3) node[below]{$\bullet$};
\draw (-2,-3) node[below]{$\bullet$};
\draw (-1,-3) node[below]{$\circ$};
\draw (4,-3) node[below]{$\circ$};
\draw (3,-3) node[below]{$\circ$};
\draw (2,-3) node[below]{$\circ$};
\draw (1,-3) node[below]{$\bullet$};
\draw (4.5,0) node[right]{$\sim$};
\end{tikzpicture}
\begin{tikzpicture}[scale=0.5]
\draw (-1.5,0) -- (-1.5,1);
\draw (-1,1) -- (-2,1);
\draw (-1,1) -- (-1,3);
\draw (-2,1) -- (-2,3);

\draw (1.5,0) -- (1.5,1);
\draw (1,1) -- (2,1);
\draw (1,1) -- (1,3);
\draw (2,1) -- (2,3);

\draw (-1.5,0) -- (-1.5,-1);
\draw (-1,-1) -- (-2,-1);
\draw (-1,-1) -- (-1,-3);
\draw (-2,-1) -- (-2,-3);

\draw (1.5,0) -- (1.5,-1);
\draw (1,-1) -- (2,-1);
\draw (1,-1) -- (1,-3);
\draw (2,-1) -- (2,-3);

\draw (-2,3) node[above]{$\bullet$};
\draw (-1,3) node[above]{$\circ$};
\draw (2,3) node[above]{$\circ$};
\draw (1,3) node[above]{$\circ$};

\draw (-2,-3) node[below]{$\bullet$};
\draw (-1,-3) node[below]{$\circ$};
\draw (2,-3) node[below]{$\circ$};
\draw (1,-3) node[below]{$\circ$};
\end{tikzpicture}
\end{center}

Note that if $p\in \CC(k, l)$ is any partition, then the points $k$ and $k+1$ of $\overline{p}\otimes p$  have different colors. We can therefore cap with a pair partition to cancel them. But then, the points $k-1$ and $k+1$ become neighbors and also have different colors, so that we can cap them again. Iterating this process, we see that we can cancel any partition of the form $\overline{p}\otimes p$ by repeated capping. We will now apply this to some general decomposition results for noncrossing partitions. Let us fix a category of noncrossing partitions $\CC$. Any projective noncrossing partition $A \in \Proj(k)$ with $t(A) = 1$ has the following form : there is a word $w = w_{0}\dots w_{n}\in \Z_{2}\ast\Z_{2}$ such that the upper part of the only through-block in $A$ has coloring $w$. Between the points colored by $w_{i}$ and $w_{i+1}$, there is a (possibly empty) partition $b_{i+1}\in \NCc(k_{i+1}, 0)$. Similarly, there are such partitions $b_{0}$ at the left of the point colored by $w_{0}$ and $b_{n+1}$ at the right of the point colored by $w_{n}$. Such a data will be symbolically written $A = [b_{0}, w_{0}, \dots, w_{n}, b_{n+1}]$ and completely characterizes the partition.

\begin{lem}\label{lem:cappingthroughblock}
Let $A = [b_{0}, w_{0}, \dots, w_{n}, b_{n+1}]$ be a projective partition as above. Then, $b_{i}^{*}b_{i}\in \CC$ for all $i$.
\end{lem}

\begin{proof}
Rotating $A$ on one line, we can cap $\overline{b}_{0}\otimes b_{0}$ to cancel it. Let $x\mapsto \overline{x}$ be the involution on $\Z_{2}$ exchanging $1$ and $-1$. Then, we get neighboring points with colors $\overline{w}_{0}$ and $w_{0}$, which we can cancel by capping again. Rotating back to get a symmetric partition, we have proven that $[b_{1}, w_{1}, \dots, w_{n}, b_{n}]\in \CC$. The same can be done on the right, and iterating this process we end up with $[b_{i}, w_{i}]\in \CC$ or $[w_{i-1}, b_{i}]\in \CC$. Rotating and capping yields $b_{i}^{*}b_{i}\in \CC$.
\end{proof}

A similar description can be given for a projective noncrossing partition $B\in \Proj(k)$ with $t(B) = 0$ and such that $1$ and $k$ belong to the same block. The coloring of the block containing $1$ can be written $w = w_{0}\dots w_{n}$ and between the points colored by $w_{i}$ and $w_{i+1}$ lies a partition $b_{i+1}\in \NCc(k_{i+1}, 0)$. Such a datum will be symbolically written $B = [w_{0}, b_{1}, \dots, b_{n}, w_{n}]$ and completely characterizes the partition.

\begin{lem}\label{lem:cappingnonthroughblock}
Let $B = [w_{0}, b_{1}, \dots, b_{n}, w_{n}]$ be a projective partition as above. Then, $b_{i}^{*}b_{i}\in \CC$ for all $i$.
\end{lem}

\begin{proof}
This is exactly the same proof as for Lemma \ref{lem:cappingthroughblock}.
\end{proof}

Using this, we can give a general decomposition result for projective noncrossing partitions.

\begin{lem}\label{lem:decomposition}
Let $\CC$ be a category of noncrossing partitions and let $p\in \CC$ be a projective partition. Then, $p$ can be (not uniquely) written as
\begin{equation*}
p = B_{0}\otimes A_{1}\otimes B_{1}\otimes A_{2}\otimes B_{2}\otimes \dots \otimes B_{t(p)-1}\otimes A_{t(p)}\otimes B_{t(p)},
\end{equation*}
where
\begin{enumerate}
\item $A_{i}$ and $B_{i}$ are projective noncrossing partitions for all $i$.
\item $t(A_{i}) = 1$ for all $i$.
\item $t(B_{i}) = 0$ for all $i$.
\end{enumerate}
Moreover, $A_{i}$ and $B_{i}$ belong to $\CC$ (note that $B_{i}$ may be empty).
\end{lem}

\begin{proof}
The existence is clear from noncrossingness and we will simply prove that the building blocks belong to $\CC$. Rotating the lower part of $B_{0}$ on the left and capping, we see as in the proof of Lemma \ref{lem:cappingthroughblock} that we can remove $B_{0}$ without leaving the category of partitions $\CC$. The same can be done for $B_{t(p)}$ by rotating it on the right and capping with a pair partition. After canceling such a partition, we can use the same rotating and capping technique to cancel $A_{1}$ or $A_{t(p)}$. It is now straightforward by induction that we can cancel $B_{0}, A_{1}, \dots, B_{i-2}, A_{i-1}$ and $B_{t(p)}, A_{t(p)}, \dots B_{i+1}, A_{i+1}$ without leaving $\CC$, i.e. $A_{i}, B_{i}\in \CC$ for all $i$.
\end{proof}

Let us say that a projective partition $A = [w_{0}, b_{1}, \dots, b_{n}, w_{n}]$ with $t(A) = 1$ as above is \emph{one-block} if $b_{i} = \emptyset$ for all $i$. One-block projective partitions are in fact enough to describe projective partitions.

\begin{prop}\label{prop:generators}
Let $A = [w_{0}, b_{1}, \dots, b_{n}, w_{n}] \in \Proj(k)$ be a projective partition such that $t(A) = 1$. Then, there exists a one-block projective partition $\widehat{A}$ and projective partitions $B, B'$ with $t(B) = 0 = t(B')$ such that $A \sim B\otimes \widehat{A}\otimes B$.
\end{prop}

\begin{proof}
If $A$ contains the same number of black and white points, it is equivalent (by capping symmetrically all the points but four) to $[\circ, \bullet]$ or to $[\bullet, \circ]$ and the result is clear. Otherwise, we can cap neighbouring points of different colors in order to get an equivalent projective partition with all points of the same color by Lemma \ref{lem:cappingprojective}. The only potential trouble is when capping points of the through-block with singletons since this removes one point of the through-block. It is however clear that in the worst case, there will only be one point left in the through-block and the result is obvious in that case. We will thus assume from now on the points to be white, the case of black points being similar. Note that, again because of singletons, the previous capping may have add non-through-block partitions on the left and on the right of the through-block but that it is of course of no consequence on the result.

Set $\beta = b_{1}^{*}b_{1}\in \Proj(l)$ (it is in $\CC$ by Lemma \ref{lem:cappingthroughblock}), let $\idpart$ denote the white identity partition and set
\begin{equation*}
R = (\beta\otimes \idpart^{\otimes (k-l)})A \in \CC.
\end{equation*}
Note that $R$ gives an equivalence between $A$ and $RR^{*}$ by Lemma \ref{lem:cappingthroughblock}. Let us number the lower points of the partitions by integers starting from the left. Then, two points $i, j\leqslant l$ are connected in $R_{l}$ if and only if they are connected in $b_{1}^{*}$ and two points $i, j \geqslant l+2$ are connected in $R_{l}$ if and only if they are connected in $A_{l}$. Moreover, the first $l$ points cannot be connected to any of the last $k-(l+1)$ points. We therefore only have to look at the point $l+1$.
\begin{itemize}
\item If in the composition defining $R$, $b_{1}$ connects $l+1$ to $1$, then $l+1$ is connected to $l+2$ in $R$ and $RR^{*} = \beta\otimes[\circ, \circ, b_{2}, \dots, \circ]$.
\item If in the composition defining $R$, $\beta$ does not connect $l+1$ to $1$, then 
$l+1$ becomes a singleton in $R_{l}$, so that $RR^{*} = \beta\otimes \beta_{1}\otimes [\circ, b_{2}, \dots, \circ]$.
\end{itemize}
We can now prove the result by induction on $n(A)$, the integer such that the through-block of $A$ has $n(A)+1$ points on each row. If $n(A)=0$, then $A$ is already one-block. If $n(A)\geqslant 1$, we have two possibilities :
\begin{itemize}
\item $A\sim B\otimes [\circ, b_{2}, \dots, \circ]$. Since $n([\circ, b_{2}, \dots, \circ]) = n(A)-1$, we can apply the induction hypothesis to $[\circ, b_{2}, \dots, \circ]$, giving the result.
\item $A\sim B\otimes [\circ, \circ, b_{2}, \dots, \circ]$ with $B\in \Proj(l)$. Let $C = [\circ, b_{2}, \dots, \circ]$. By induction, $C\sim B'\otimes \widehat{C}$. Let $D$ be the partition implementing this equivalence and consider the partition
\begin{equation*}
R' = (\idpart^{\otimes l+1}\otimes D)(B\otimes [\circ, \circ, b_{2}, \dots, \circ]).
\end{equation*}
This is an equivalence between $B\otimes [\circ, \circ, b_{2}, \dots, \circ]$ and $[\circ, B', \circ, \circ, \dots, \circ]$. Applying again our construction, we can get $B'$ out of the through-block so that $A\sim B\otimes B'\otimes [\circ, \circ, \dots, \circ]$.
\end{itemize}
\end{proof}

Here is an instance of such an equivalence :
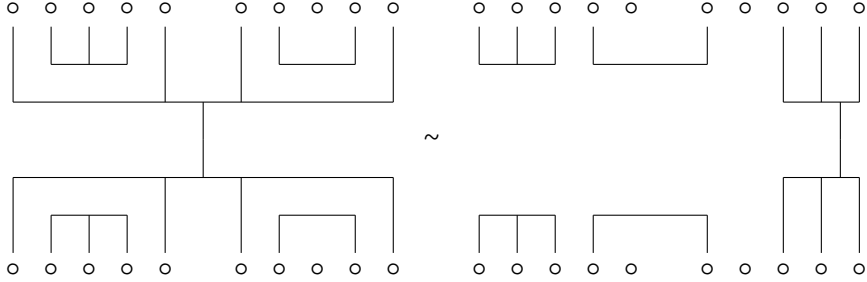
\begin{center}
\begin{tikzpicture}[scale=0.5]
\draw (0,0) -- (0, 1);
\draw (-5,1) -- (5,1);
\draw (-5,1) -- (-5, 3);
\draw (-4,2) -- (-4, 3);
\draw (-3,2) -- (-3, 3);
\draw (-2,2) -- (-2, 3);
\draw (-4,2) -- (-2, 2);
\draw (-1,1) -- (-1, 3);
\draw (5,1) -- (5, 3);
\draw (4,2) -- (4, 3);
\draw (2,2) -- (2, 3);
\draw (4,2) -- (2, 2);
\draw (1,1) -- (1, 3);

\draw (0,0) -- (0,-1);
\draw (-5,-1) -- (5,-1);
\draw (-5,-1) -- (-5, -3);
\draw (-4,-2) -- (-4, -3);
\draw (-3,-2) -- (-3, -3);
\draw (-2,-2) -- (-2, -3);
\draw (-4,-2) -- (-2, -2);
\draw (-1,-1) -- (-1, -3);
\draw (5,-1) -- (5, -3);
\draw (4,-2) -- (4, -3);
\draw (2,-2) -- (2, -3);
\draw (4,-2) -- (2, -2);
\draw (1,-1) -- (1, -3);

\draw (-5,3) node[above]{$\circ$};
\draw (-4,3) node[above]{$\circ$};
\draw (-3,3) node[above]{$\circ$};
\draw (-2,3) node[above]{$\circ$};
\draw (-1,3) node[above]{$\circ$};
\draw (5,3) node[above]{$\circ$};
\draw (4,3) node[above]{$\circ$};
\draw (3,3) node[above]{$\circ$};
\draw (2,3) node[above]{$\circ$};
\draw (1,3) node[above]{$\circ$};

\draw (-5,-3) node[below]{$\circ$};
\draw (-4,-3) node[below]{$\circ$};
\draw (-3,-3) node[below]{$\circ$};
\draw (-2,-3) node[below]{$\circ$};
\draw (-1,-3) node[below]{$\circ$};
\draw (5,-3) node[below]{$\circ$};
\draw (4,-3) node[below]{$\circ$};
\draw (3,-3) node[below]{$\circ$};
\draw (2,-3) node[below]{$\circ$};
\draw (1,-3) node[below]{$\circ$};

\draw (5.5,0) node[right]{$\sim$};
\end{tikzpicture}
\begin{tikzpicture}[scale=0.5]
\draw (4.5,0) -- (4.5, 1);
\draw (3,1) -- (5,1);
\draw (-2,2) -- (1, 2);
\draw (-3,2) -- (-5, 2);
\draw (-3,2) -- (-3, 3);
\draw (-4,2) -- (-4, 3);
\draw (-5,2) -- (-5, 3);
\draw (-2,2) -- (-2, 3);
\draw (1,2) -- (1, 3);
\draw (3,1) -- (3, 3);
\draw (4,1) -- (4, 3);
\draw (5,1) -- (5, 3);

\draw (4.5,0) -- (4.5, -1);
\draw (3,-1) -- (5,-1);
\draw (-2,-2) -- (1, -2);
\draw (-3,-2) -- (-5, -2);
\draw (-3,-2) -- (-3, -3);
\draw (-4,-2) -- (-4, -3);
\draw (-5,-2) -- (-5, -3);
\draw (-2,-2) -- (-2, -3);
\draw (1,-2) -- (1, -3);
\draw (3,-1) -- (3, -3);
\draw (4,-1) -- (4, -3);
\draw (5,-1) -- (5, -3);

\draw (-5,3) node[above]{$\circ$};
\draw (-4,3) node[above]{$\circ$};
\draw (-3,3) node[above]{$\circ$};
\draw (-2,3) node[above]{$\circ$};
\draw (-1,3) node[above]{$\circ$};
\draw (5,3) node[above]{$\circ$};
\draw (4,3) node[above]{$\circ$};
\draw (3,3) node[above]{$\circ$};
\draw (2,3) node[above]{$\circ$};
\draw (1,3) node[above]{$\circ$};

\draw (-5,-3) node[below]{$\circ$};
\draw (-4,-3) node[below]{$\circ$};
\draw (-3,-3) node[below]{$\circ$};
\draw (-2,-3) node[below]{$\circ$};
\draw (-1,-3) node[below]{$\circ$};
\draw (5,-3) node[below]{$\circ$};
\draw (4,-3) node[below]{$\circ$};
\draw (3,-3) node[below]{$\circ$};
\draw (2,-3) node[below]{$\circ$};
\draw (1,-3) node[below]{$\circ$};
\end{tikzpicture}
\end{center}

Let us give a corollary summarizing the results of this section.

\begin{cor}\label{cor:generators}
Let $p\in \Proj$ be a projective partition. Then, there are one-block partitions $A_{1}, \dots, A_{t(p)}\in \CC$ and projective partitions $B_{0}, \dots, B_{t(p)}\in \CC$ with $t(B_{i}) = 0$ such that
\begin{equation*}
p \sim B_{0}\otimes A_{1}\otimes B_{1}\otimes \dots \otimes A_{t(p)}\otimes B_{t(p)}.
\end{equation*}
\end{cor}

\subsection{Block-stability}

Our main concern in this subsection is to understand the operation of passing from a partition to a subpartition and in particular to a block. More precisely, we will see that the possibility of passing to blocks imposes strong conditions on a category of partitions. We sart with a natural definition :

\begin{de}
A category of partitions $\CC$ is said to be \emph{block-stable} if for any partition $p\in \CC$ and any block $b$ of $p$, we have $b\in \CC$.
\end{de}

Not all categories of partitions are block-stable (even not all categories of noncrossing partitions). In fact, Theorem \ref{thm:mainresult} gives a characterization of block-stable categories of noncrossing partitions. For simplicity, let us give a companion definition.

\begin{de}
Let $\CC$ be a fixed category of partitions. A partition $p\in \CC$ is said to be \emph{block-stable} if any block of $p$ is in $\CC$.
\end{de}

\begin{rem}
The notion of block stability makes no sense for a general partition $p$ if a category of partitions is not specified. In the sequel, the category which is referred to will always be clear. Note that a category of partitions $\CC$ is block-stable if and only if all its partitions are block-stable.
\end{rem}

Let us give an elementary property of block-stable noncrossing partitions.

\begin{lem}\label{lem:blockstablepartition}
Let $\CC$ be a category of partitions and let $p\in \CC$ be a block-stable \emph{noncrossing} partition with $t(p) = 0$. Then, $p_{u}$ and $p_{l}$ both belong to $\CC$.
\end{lem}

\begin{proof}
First note that intervals of $p$, i.e. blocks of the form $\{i, i+1, \dots, i+l\}$ can be removed by capping since they belong to $\CC$ by assumption. This creates new intervals, which can also be removed. Because $p$ is noncrossing, iterating this process enables us to remove all of $p_{l}$ or all of $p_{u}$ without leaving $\CC$.
\end{proof}

Proposition \ref{prop:nontrivial1d} will prove crucial in our investigation of the link between block-stability and the representation theory of the associated easy quantum group. Before stating and proving it, we need a preparatory lemma.

\begin{lem}\label{lem:nontrivial1d}
Let $\CC$ be a category of noncrossing partitions and assume that it is \emph{not} block-stable. Then, there exists a partition $p\in \CC$ satisfying
\begin{enumerate}
\item $p$ is not block-stable.
\item $p$ is projective.
\item $t(p) = 0$.
\item $1$ and $k$ belong to the same block of $p$.
\end{enumerate}
\end{lem}

\begin{proof}
Let $r$ be a partition which is not block-stable and let $\widehat{r}$ be the partition obtained by rotating all the points of $r$ on the upper line. Then, $p = (\widehat{r})^{*}\widehat{r}$ is projective, is not block-stable and $t(p) = 0$. Assume that $1$ and $k$ do not belong to the same block of $p$. This means (by noncrossingness) that $p$ can be written as $q_{1}\otimes q_{2}$, where $q_{1}$ and $q_{2}$ are projective partition with $t(q_{i}) = 0$ for $i=1, 2$. Any block of $p$ being either a block of $q_{1}$ or of $q_{2}$, at least one of them, say $q_{1}$, is not block-stable. By rotating $q_{2}$ on one line and capping, we see that $q_{1}\in \CC$. Iterating this process, we end up with a partition satisfying condition $(4)$.
\end{proof}

\begin{prop}\label{prop:nontrivial1d}
Let $\CC$ be a category of noncrossing partitions and assume that it is \emph{not} block-stable. Then, there exists a partition $b\in \CC(k, 0)$ such that $b^{*}b\in \CC$ but $b\notin \CC$.
\end{prop}

\begin{proof}
Let $p\in \Proj(k)$ be given by Lemma \ref{lem:nontrivial1d} and write $p = [w_{1}, b_{1}, \dots, b_{n}, w_{k}]$. Let us prove that either $p_{u}\notin \CC$ or there exists a partition $q\in \Proj(l)$, $l<k$, satisfying the properties $(1)$--$(4)$ of Lemma \ref{lem:nontrivial1d}. If $p_{u}\in \CC$, there are in fact two possibilities : all the partitions $b_{i}$ are block-stable, or one of them is not.

\begin{enumerate}
\item In the first case, capping by blocks of these partitions we can remove all of them. We end up with a projective partition, the upper row of which is simply a $l$-block for some $l<k$ (because $p_{u}\in \CC$ by assumption, so that this block cannot be all of $p_{u}$). This $l$-block is not in $\CC$ because all the other blocks of $p$ are blocks of some $b_{i}$, hence in $\CC$. Thus, we are done.
\item In the second case, there is an index $i$ such that $b_{i}$ contains a block which is not in $\CC$. Then, we know by Lemma \ref{lem:cappingnonthroughblock} that $b_{i}^{*}b_{i}\in \CC$. Since $b_{i}$ has by definition strictly less points than $p$, we are done.
\end{enumerate}

Now it is clear by induction that there is a projective partition $q\in \CC$ such that $t(q) = 0$ and $q_{u}\notin \CC$. Setting $b = q_{u}$ concludes the proof.
\end{proof}

We end this section with elementary properties of the conjugation and $\boxvert$ operations on projective partitions.

\begin{lem}\label{lem:conjugation}
Let $\CC$ be a \emph{block-stable} category of partitions and let $p, q\in \Proj$ be such that $p\square^{t(p)}q\in \CC$ and $t(p) = t(q)$. Then, $q\sim \overline{p}$.
\end{lem}

\begin{proof}
Note that $p\square^{t(p)}q$ is a rotated version of $r^{\overline{p}}_{q}\otimes r^{\overline{q}}_{p}$, hence $r^{\overline{p}}_{q}\in \CC$ by block-stability, concluding the proof.
\end{proof}

\begin{lem}\label{lem:fusion}
Let $\CC$ be a category of noncrossing partitions. Let $a, b, a', b'\in \CC$ be projective partitions such that $a\sim a'$ and $b\sim b'$. If $a\boxvert b \in \CC$, then $a'\boxvert b'\in \CC$ and these two partitions are equivalent.
\end{lem}

\begin{proof}
We have
\begin{equation*}
r^{a\boxvert b}_{a'\boxvert b'} = (r^{a}_{a'}\otimes r^{b}_{b'})(a\boxvert b) \in \CC.
\end{equation*}
\end{proof}

This leads to the following definition :

\begin{de}\label{de:fusion}
Let $\CC$ be a category of noncrossing partitions and let $p$ and $q$ be projective partitions in $\CC$. Then, we denote by $[p]\ast[q]$ the equivalence class of the partition $p\boxvert q$ if the latter partition is in $\CC$. According to Lemma \ref{lem:fusion}, this is a well-defined equivalence class. If $p\boxvert q \notin \CC$, we set $[p]\ast[q] = \emptyset$.
\end{de}

\subsection{Characterization of free fusion rings}

In order to state and prove our main result, let us give some details conerning the construction of the free fusion semiring $\RR$ introduced in \cite[Sec 6.4]{freslon2013representation}. Starting with a category of noncrossing partitions $\CC$, we form the set $S(\CC)$ of equivalence classes of one-block projective partitions. This set is endowed with the \emph{conjugation map}
\begin{equation*}
[p]\mapsto \overline{[p]} = [\overline{p}]
\end{equation*}
(it is clear that $p\sim q$ if and only if $\overline{p}\sim \overline{q}$) and with the \emph{fusion operation}
\begin{equation*}
([p], [q])\mapsto [p]\ast[q]
\end{equation*}
of Definition \ref{de:fusion}. We then build out the \emph{free fusion semiring} associated to $(S(\CC), -, \ast)$ as in Definition \ref{de:fusionsemiring} and denote it $(\RR, \oplus, \otimes)$. The goal of this section is to understand the link between this fusion semiring and the fusion semiring of the associated easy quantum group. More precisely, we will be interested in the map
\begin{equation*}
\Phi : \RR \rightarrow R^{+}(\G)
\end{equation*}
sending a word $[p_{1}]\dots [p_{n}]$ to $[u_{p_{1}\dots p_{n}}]$ and extended by linearity (note that this map is well-defined). Let us study its set-theoretic properties.

\begin{lem}\label{lem:bijectivePhi}
Let $\CC$ be a \emph{block-stable} category of noncrossing partitions. Then, the map $\Phi$ is bijective.
\end{lem}

\begin{proof}
The proof of the injectivity of $\Phi$ was sketched in \cite[Lem 6.13]{freslon2013representation}, but we give a more detailed argument. Let $w = [p_{1}]\dots [p_{n}]$ and $w' = [q_{1}]\dots [q_{k}]$ be words on $S(\CC)$ such that $\Phi(w) = \Phi(w')$. This means that the projective partitions $p = p_{1}\otimes \dots \otimes p_{n}$ and $q = q_{1}\otimes \dots \otimes q_{k}$ are equivalent. Since $t(p) = n$ and $t(q) = k$, we must have $n=k$. Setting $r_{i} = r^{p_{i}}_{q_{i}}$, we see that $r^{p}_{q} = q_{u}^{*}p_{u} = r_{1}\otimes \dots \otimes r_{n}$. By block-stability, $r_{i}\in \CC$ for all $i$, i.e. $p_{i}\sim q_{i}$ for all $i$. Hence, $w = w'$ in $S(\CC)$ and $\Phi$ is injective.

Let us now prove surjectivity. Block stability means in particular that we can cancel all the non-through-blocks in a the decomposition of Corollary \ref{cor:generators} without changing its equivalence class. By Corollary \ref{cor:generators}, any projective partition is therefore equivalent to an horizontal concatenation of one-block projective through-partitions. In other words, $\Phi$ is surjective.
\end{proof}

The last ingredient we need is some precision about the notion of \emph{trivial representation} which will be important hereafter.

\begin{lem}\label{lem:trivial}
Let $p$ be a projective partition such that $t(p) = 0$. Then, $u_{p}$ is a one-dimensional representation. Moreover, $u_{p}$ is equivalent to the trivial representation of $\G$ if and only if $p_{u}\in \CC$.
\end{lem}

\begin{proof}
Because $t(p) = 0$, $T_{p}$ has rank one so that $P_{p} = T_{p}$ and $u_{p}$ is one-dimensional. Consider now a partition $b\in \CC$ lying on one line. Then, $b^{*}(b\otimes b)\in \CC$, so that $b^{*}b\sim (b^{*}b)\otimes (b^{*}b)$. Setting $u = u_{b^{*}b}$, we get $u\sim u\otimes u$ and tensoring with the contragredient representation gives $u\sim \varepsilon$, where $\varepsilon$ denotes the trivial representation of $\G$. Thus, $u_{p}$ is trivial if and only if $b^{*}p_{u} = (b^{*}b)_{l}p_{u} = r^{p}_{b^{*}b}\in \CC$. Since $b\in \CC$, this is equivalent to $p_{u}\in \CC$.
\end{proof}

We are now ready for our main result.

\begin{thm}\label{thm:mainresult}
Let $\CC$ be a category of noncrossing partitions, let $N\geqslant 4$ be an integer and let $\G$ be the associated easy quantum group. The following are equivalent :
\begin{enumerate}
\item $\G$ has no nontrivial one-dimensional representation.
\item $\CC$ is block-stable.
\item The map $\Phi$ is a semiring isomorphism.
\item The fusion semiring of $\G$ is free.
\end{enumerate}
\end{thm}

\begin{proof}
$(1) \Rightarrow (2)$ : Assume that $\CC$ is not block-stable. Then, by Proposition \ref{prop:nontrivial1d}, there is a partition $b\notin\CC$ lying on one line such that $b^{*}b\in \CC$. Thus, $u_{b^{*}b}$ is a one-dimensional representation which is not equivalent to the trivial one by Lemma \ref{lem:trivial}.

$(2) \Rightarrow (3)$ : If $\CC$ is block-stable, then $\Phi$ is bijective by Lemma \ref{lem:bijectivePhi}. Let $w = [p_{1}]\dots [p_{n}]$ and set $p = p_{1}\otimes \dots \otimes p_{n}$. By definition of the conjugation on $S(\CC)$, $\Phi(w)\otimes \Phi(\overline{w}) = [u_{p}\otimes u_{\overline{p}}]$. The representation $u_{p}\otimes u_{\overline{p}}$ contains $u_{p\square^{t(p)}\overline{p}}$ and since $(p\square^{t(p)}\overline{p})_{u}$ is a rotation of $p$, the latter representation is trivial by Lemma \ref{lem:trivial}. As $\Phi(\overline{w})$ is an equivalence class of irreducible representations, it is the class of the contragredient of $\Phi(w)$ and $\Phi$ preserves the conjugation operation. We can now prove that $\Phi$ respects tensor products. Let $w = [q_{1}]\dots [q_{n}]$ and let $k$ be an integer such that $(p_{1}\dots p_{n})\square^{k} (q_{1}\dots q_{n})\in \CC$. By block stability, we see that $p_{n-i}\square q_{i+1}\in \CC$ for every $0\leqslant i\leqslant k-1$. This has two consequences :
\begin{itemize}
\item $[p_{n-i}] = \overline{[q_{i+1}]}$ by Lemma \ref{lem:conjugation}. Moreover, setting $z = [p_{n-k+1}]\dots [p_{n}]$, we have $w = az$ and $w' = \overline{z}b$.
\item $[u_{(p_{1}\dots p_{n})\square^{k} (q_{1}\dots q_{n})}] = \Phi(ab)$.
\end{itemize}
Similarly, if $(p_{1}\dots p_{n})\boxvert^{k} (q_{1}\dots q_{n})\in \CC$ then there is a unique $z$ of length $k-1$ such that $w = az$, $w' = \overline{z}b$ and $[u_{(p_{1}\dots p_{n})\boxvert^{k} (q_{1}\dots q_{n})}] = \Phi(a\ast b)$.

$(3) \Rightarrow (4)$ : $(\RR, \oplus, \otimes)$ is by definition a free fusion semiring.

$(4) \Rightarrow (1)$ : Let $u$ be a nontrivial one-dimensional representation. Then, $\overline{u}\otimes u = \varepsilon$ in $R^{+}(\G)$ but $u\neq \varepsilon$, hence the fusion semiring is not free (see \cite[Rem 4.4]{raum2012isomorphisms}).
\end{proof}

As a corollary, we can now give a converse to Lemma \ref{lem:bijectivePhi}.

\begin{cor}
Let $\CC$ be a category of noncrossing partitions. Then, $\Phi$ is surjective if and only if $\CC$ is block-stable.
\end{cor}

\begin{proof}\label{rem:bijectivephi}
It was proved in Lemma \ref{lem:bijectivePhi} that $\Phi$ is bijective if $\CC$ is block-stable. Assume conversely that $\CC$ is not block-stable. Then, Proposition \ref{prop:nontrivial1d} provides us with a partition $b\notin\CC$ such that $b^{*}b\in \CC$ and $t(b^{*}b) = t(b) = 0$. In particular, $u_{b^{*}b}$ cannot be equivalent to a representation $u_{p}$ if $t(p)\geqslant 1$. Since any representation in the image of $\Phi$ is equivalent to $u_{p}$ for some projective partition $p$ satisfying $t(p)\geqslant 1$, $\Phi$ is not surjective.
\end{proof}

\subsection{Classification}

A possible interpretation of Theorem \ref{thm:mainresult} is that $R^{+}(\G)$ contains a "free part" $\RR$, to which it reduces precisely when it is free. We therefore now want to get a better understanding of the set $S(\CC)$ and its fusion operation. This will in particular lead us to a classification of all the \emph{free} fusion semirings arising from easy quantum groups. To do this, let us first give an alternative description of $\RR$. Let $k\geqslant 1$ be an integer, let $\pi_{k}$ denote the unique one-block partition in $\NCc(k, k)$ with all points colored in white and set $\pi_{-k} = \overline{\pi}_{k}$ (i.e. all the points are colored in black). We first consider the objects
\begin{equation*}
\left\{\begin{array}{ccc}
I(\CC) & = & \{k\in \Z^{*}, \pi_{k}\in \CC\} \\
k\sim k' & \Leftrightarrow & \pi_{k}\sim \pi_{k'}
\end{array}\right.
\end{equation*}

\begin{lem}\label{lem:classificationfusionring}
Let $\CC$ be a category of noncrossing partitions. Then $I(\CC)$ is equal either to $\Z^{*}$ or to $\{-1, 1\}$.
\end{lem}

\begin{proof}
First, $\pi_{1}\in I(\CC)$ by definition of a category of partitions. Rotating $\pi_{k}$ upside down yields $\pi_{-k}$, hence $I(\CC)$ is symmetric. Assume now that $\pi_{k}$ is in $\CC$. Then, rotating it on one line and capping in the middle with a block of size $2$, we get one block with $k-1$ white points and $k-1$ black points, i.e. a rotated version of $\pi_{k-1}$. Hence, if $k\in I(\CC)$ and $k\geqslant 2$ then $k-1\in I(\CC)$. Assume now that $\pi_{2}\in \CC$ and note that
\begin{equation*}
(\pi_{2k}\otimes \pi_{2k})(\pi_{1}^{\otimes k}\otimes \pi_{2k}\otimes \pi_{1}^{\otimes k})(\pi_{2k}\otimes \pi_{2k}) = \pi_{4k}.
\end{equation*}
Therefore $I(\CC) = \Z^{*}$ as soon as $2\in I(\CC)$, concluding the proof.
\end{proof}

The set $I(\CC)$ together with its equivalence relation encodes almost the same information as $S(\CC)$. In fact, if $[k]$ denotes the equivalence class of $k$ for the relation $\sim$, we have $\overline{[k]} = [-k]$ and the fusion operation is given by
\begin{equation*}
[k]\ast[k'] = \left\{\begin{array}{ccc} 
[k+k'] & \text{if} & k+k'\in I(\CC) \\
\emptyset & \text{if} & k+k' \neq 0 \text{ and } k+k'\neq I(\CC)
\end{array}\right.
\end{equation*}
As we see, the only thing we need to recover $S(\CC)$ is a "$0$ element". However, there is a subtlety at that point :  there are two (a priori) distinct zero elements. More precisely, let $\pi_{0^{+}} = [\circ, \bullet]$ be a one-block projective partition in $\NCc(2, 2)$ (i.e. of the form $\vierpartrot$) and set $\pi_{0^{-}} = [\bullet, \circ]$ in the same way. The following facts are straightforward.

\begin{lem}\label{lem:0equivalence}
We have $\overline{\pi}_{0^{+}}\sim \pi_{0^{+}}$, $\overline{\pi}_{0^{-}}\sim \pi_{0^{-}}$ and $\pi_{0^{+}}\in \CC$ if and only if $\pi_{0^{-}}\in \CC$. Moreover, $\pi_{0^{+}}\sim \pi_{0^{-}}$ if and only if $\pi_{2}\in \CC$.
\end{lem}

If $\pi_{0^{+}}\in\CC$, then we have, for any $k\geqslant 0$,
\begin{equation*}
\left\{\begin{array}{ccc}
\left[k\right]\ast \left[-k\right] & = & [0^{+}] \\
\left[-k\right]\ast \left[k\right] & = & [0^{-}]
\end{array}\right.
\end{equation*}

We can now give another description of $\RR$. Consider the set
\begin{equation*}
I'(\CC) = \left\{\begin{array}{cc}
I(\CC)\cup \{0^{+}, 0^{-}\}  & \text{if } \pi_{0^{+}}\in \CC \\
I(\CC) & \text{otherwise}
\end{array}\right.
\end{equation*}
Capping neighboring blocks of different colors repeatedly, we see that any one-block projective partition is equivalent to $\pi_{x}$ for some $x\in I'(\CC)$. Thus, we have
\begin{equation*}
S(\CC) = I'(\CC)/\sim,
\end{equation*}
with the involution given by the opposite integer and the fusion given by the addition (with the special rule for $0^{\pm}$). Using this, we can classify the free fusion rings $\RR$ arising from categories of noncrossing partitions. According to Theorem \ref{thm:mainresult}, this gives in particular all the possible free fusion semirings $R^{+}(\G)$ of easy quantum groups. As will appear, there is one case where the description is a bit intricate. Let us introduce it now to simplify further reference.

\begin{de}
Let $\mathcal{S}$ be the set with four elements $\{\alpha, \beta, \gamma, \overline{\gamma}\}$ endowed with the involution $\overline{\alpha} = \alpha$, $\overline{\beta} = \beta$ and the fusion operations
\begin{equation*}
\begin{array}{cccccc}
\gamma \ast \overline{\gamma} & = & \alpha & \overline{\gamma}\ast \gamma & = & \beta \\
\alpha\ast \alpha & = & \alpha & \beta\ast \beta & = & \beta \\
\gamma\ast\gamma = \overline{\gamma}\ast\overline{\gamma} & = & \emptyset & \alpha\ast \beta = \beta \ast \alpha & = & \emptyset \\
\gamma\ast\alpha = \overline{\gamma}\ast\beta & = & \emptyset & \gamma\ast\beta = \alpha\ast\gamma & = & \gamma \\
\overline{\gamma}\ast\alpha  = \beta\ast \overline{\gamma} & = & \gamma & \beta\ast \gamma =  \alpha\ast \overline{\gamma} & = & \emptyset \\
\end{array}
\end{equation*}
\end{de}

\begin{thm}\label{thm:classificationfree}
Let $\CC$ be a category of noncrossing partitions. Then,
\begin{enumerate}
\item If $\pi_{0^{+}}\notin \CC$, then $S(\CC) = \{[1]\}$ or $S(\CC) = \{[-1], [1]\}$.
\item If $\pi_{0^{+}}\in \CC$, then $S(\CC) = \mathcal{S}$ or $S(\CC) = \Z_{s}$ for some integer $\infty\geqslant s\geqslant 1$.
\end{enumerate}
\end{thm}

\begin{proof}
Assume that $\pi_{0^{+}}\notin \CC$. Since capping $\pi_{2}\otimes \pi_{-2}$ yields $\pi_{0^{+}}$, we must have, by Lemma \ref{lem:classificationfusionring}, $I(\CC) = \{-1, 1\}$. There are then only two possible equivalence relations : either $1$ is equivalent to $-1$ (yielding the first case) or $1$ is not equivalent to $-1$ (yielding the second case).

Assume now that $\pi_{0^{+}}\in \CC$ and that $I(\CC) = \{-1, 1\}$. Rotating and capping $r^{\pi_{0^{+}}}_{\pi_{\pm 1}}$ we get a singleton, which, combined again with $\pi_{0^{+}}$ would imply that we can change the colors of any partition and thus that $I(\CC) = \Z^{*}$ (because then $\pi_{2}\in \CC$), a contradiction. In other words, $\pi_{0^{+}}$ cannot be equivalent to $\pi_{\pm 1}$. The same happens if we assume $\pi_{0^{+}} \sim \pi_{0^{-}}$ (because $r^{\pi_{0^{+}}}_{\pi_{0^{-}}}$ is a rotated version of $\pi_{2}$). Therefore, there are $4$ equivalence classes in $S(\CC)$. It is clear that the conjugation map and the first six equations giving the fusion operation are that of $\mathcal{S}$ under the identification $\alpha = 0^{+}$, $\beta = 0^{-}$ and $\gamma = 1$. To see that the last four ones are also satisfied, notice for instance that $\pi_{1}\ast\pi_{0^{+}}\in \CC$ implies that $\pi_{2}\in \CC$, contradicting $I(\CC) = \{-1, 1\}$. The other cases are done similarly.

Assume eventually that $\pi_{0^{+}}\in \CC$ and that $I(\CC) = \Z^{*}$. Then, $\pi_{0^{+}} \sim \pi_{0^{-}}$ by Lemma \ref{lem:0equivalence} and $S(\CC)$ is a quotient (as an additive group) of $\Z$, i.e. $S(\CC) = \Z_{s}$ for some $s$.
\end{proof}

As we will see later on, Theorem \ref{thm:classificationfree} is complete in the sense that there are categories of noncrossing partitions which are block-stable and yield all the possible free fusion semirings. Let us list them now, even though proofs will be postponed to the next section :
\begin{equation*}
\begin{array}{|c|c|} \hline
S(\mathcal{C}^{\circ, \bullet}) & \mathbb{G} \\ \hline
\{[1]\} & O_{N}^{+} \\ \hline
\{[-1], [1]\} & U_{N}^{+} \\ \hline
\mathbb{Z}_{s}, 1\leqslant s\leqslant \infty & H_{N}^{s+} \text{ (note that }H^{1+}_{N} = S_{N}^{+}\text{)} \\ \hline
\mathcal{S} & \widetilde{H}_{N}^{+} \text{ (see Definition \ref{de:alternating})}\\ \hline
\end{array}
\end{equation*}

We are not claiming that the only block-stable categories of noncrossing partitions are those corresponding to the above quantum groups. In fact, the category of noncrossing partitions $\mathcal{B}^{\circ, \bullet} = \langle \theta_{1}\rangle$ is obviously block-stable and $S(\mathcal{B}^{\circ, \bullet}) = \{[-1], [1]\}$. However, the associated quantum group cannot be isomorphic to $U_{N}^{+}$ since it has an irreducible representation of dimension $N-1$ (but it can be seen to be isomorphic to $U_{N-1}^{+}$). However, the classification of all categories of noncrossing partitions which is currently undergone by P. Tarrago et M. Weber \cite{tarrago2015unitary} will straightforwardly yield the list of all block-stable categories of noncrossing partitions. We thank the authors for having kindly communicated to us part of their results.

\subsection{Examples}

We will now show how Theorem \ref{thm:mainresult} applies to the quantum reflection groups $H_{N}^{s+}$ for $\infty > s\geqslant 1$. The fusion rules of these quantum groups were studied in \cite{banica2009fusion} and, as one expects, our technique recovers the results of this paper in a very natural way : the set $I(\CC)$ is equal to $\Z^{*}$ and the equivalence relation $\sim$ is equality modulo $s$. To see this, we first have to describe the "easy structure" of $H_{N}^{s+}$, i.e. its category of partitions. Let us denote by $\theta_{s}\in \NCc(k, 0)$ the one-block partition with all points colored in white. We then define, for $s\geqslant 1$, a category of partitions $\CC_{s} = \langle \pi_{2}, \theta_{s}\rangle$.

\begin{rem}\label{rem:simplecases}
When, $s\geqslant 3$, $\pi_{2}$ can be constructed out of $\theta_{s}$ using the category operations. The presence of $\pi_{2}$ only ensures that when $s=1$, we recover the quantum permutation group $S_{N}^{+}$ and when $s=2$, we recover the free hyperoctahedral quantum group $H_{N}^{+}$. This is straightforward to prove.
\end{rem}

\begin{prop}
Let $N\geqslant 4$ be an integer and let $s\geqslant 1$. Then, the easy unitary quantum group $\G$ associated to $\CC_{s}$ is the quantum group $H_{N}^{s+}$ of \cite[Def 1.3]{banica2009fusion}.
\end{prop}

\begin{proof}
In view of Remark \ref{rem:simplecases}, we can assume $s\geqslant 3$. It is proved in \cite[Thm 6.3]{banica2009fusion} that the category of partitions associated with the quantum group $H_{N}^{s+}$ is the category of noncrossing partitions satisfying the following property : in each block, the difference between the number of white and black points on each row is the same modulo $s$. In particular, it contains $\theta_{s}$ and there is a surjective map
\begin{equation*}
C_{\text{max}}(\G)\longrightarrow C_{\text{max}}(H_{N}^{s+})
\end{equation*}
sending the fundamental representation onto the fundamental representation. To prove that this map is an isomorphism, let us first make some manipulations. Capping $\theta_{s}\otimes \overline{\theta}_{s}\otimes \theta_{s}$ twice, we get a one-block partition with $s-1$ white points followed by $s-2$ black points and $s-1$ white points again. Using a rotated version of $\theta_{s}$ to change the $s-2$ black points into $2$ white points, we see that $\theta_{2s}\in\CC_{s}$. More generally, $\theta_{ks}\in \CC_{s}$ for any integer $k$. Now, we can again use a rotated version of $\theta_{s}$ to change the last $s-1$ white points of $\theta_{2s}$ in black and obtain a partition $p$. The fact that $T_{p}$ is an intertwiner exactly means that the coefficients $u_{i, j}$ of the fundamental representation $u$ satisfy
\begin{equation*}
u_{i, j}^{s} = u_{i, j}u_{i, j}^{*}.
\end{equation*}
Using similar techniques, one can build out of $\pi_{2}$ a partition which implies that $u_{i, j}u_{i, j}^{*}$ is a projection. Applying the definition of $H_{N}^{s+}$ \cite[Def 1.3]{banica2009fusion}, we therefore get a surjective map
\begin{equation*}
C_{\text{max}}(H_{N}^{s+})\longrightarrow C_{\text{max}}(\G)
\end{equation*}
sending the fundamental representation onto the fundamental representation. This property implies that this map is the inverse of the previous one, hence the result.
\end{proof}

\begin{prop}
For any $\infty > s\geqslant 1$, we have $S(\CC_{s}) = \Z_{s}$ (with $\Z_{1} = \{1\}$).
\end{prop}

\begin{proof}
The proofs of the cases $s = 1$ and $s = 2$ were done in \cite[Sec 5.2]{freslon2013representation}, so that we may assume $s\geqslant 3$. Let us set $\pi_{0} = \pi_{0^{+}}$, which is equivalent to $\pi_{0^{-}}$. The proposition follows from the following elementary facts :
\begin{itemize}
\item The only partitions of the form $\theta_{x}$ in $\CC_{s}$ are exactly those where $x$ is a multiple of $s$ (use the description of $\CC_{s}$ in terms of number of white and black points).
\item Capping in the middle of $\theta_{s}\otimes \overline{\theta}_{s}$ yields a rotated version of $\pi_{s-1}$. Hence, $I(\CC_{s}) = \Z^{*}$.
\item Rotating $\theta_{s}$, we get an equivalence between $\pi_{k}$ and $\pi_{k-s}$ for any $0 < k < s$.
\item Capping $\pi_{k+s}$ with $\theta_{s}$ gives an equivalence with $\pi_{k}$ for any $k>0$. Rotating gives the corresponding statement for negative integers.
\item Let $\theta'_{s}$ be the partition obtained by rotating one point of $\theta_{s}$ to the lower row. Then, $\pi_{0}(\pi_{1}\otimes \theta'_{s})$ gives an equivalence between $\pi_{0}$ and $\pi_{s}$.
\item Reciprocally, $\pi_{k}\sim \pi_{k'}$ implies, by rotating $r^{\pi_{k}}_{\pi_{k'}}$, that $\theta_{\vert k - k'\vert}\in \CC_{s}$, hence $\vert k - k'\vert$ must be a multiple of $s$.
\end{itemize}
\end{proof}

Note that this proposition gives an alternative proof of \cite[Thm 7.3]{banica2009fusion}. We can also treat the case of $H_{N}^{\infty+}$ along the same lines : set $\CC_{\infty} = \langle \pi_{2} \rangle$. The category $\CC_{\infty}$ can be alternatively described by the following property : this is the category of all noncrossing partitions such that in each block, the difference between the number of white and black points on each row is the same. According to \cite[Thm 6.3]{banica2009fusion}, this gives rise to the inifinite hyperoctahedral quantum group $H_{N}^{\infty +}$ for $N\geqslant 4$. Applying the same reasoning as before proves that $I(\CC_{\infty}) = \Z^{*}$ and $S(\CC_{\infty}) = \Z$, giving back the fusion rules computed in \cite[Thm 7.3]{banica2009fusion}.

The free fusion semiring associated to $\mathcal{S}$ also corresponds to an hyperoctahedral quantum group, though different from the previous ones. As we will see, it corresponds to the \emph{free complexification} $\widetilde{H}_{N}^{+}$ of the free hyperoctahedral quantum group $H_{N}^{+}$.

\begin{de}\label{de:alternating}
Let $\CC_{0^{+}}$ be the category of partitions generated by $\pi_{0^{+}}$.
\end{de}

It is clear that this quantum group has no nontrivial one-dimensional representation and that $S(\CC_{0^{+}}) = \mathcal{S}$. This quantum group can also be described through its maximal C*-algebra.

\begin{de}
Let $A_{h}^{0^{+}}(N)$ be the universal C*-algebra generated by the coefficient $(u_{ij})_{1\leqslant i, j\leqslant N}$ of a matrix $u$ such that :
\begin{itemize}
\item The matrices $u$ and $\overline{u}$ are unitary.
\item For every $1\leqslant k\leqslant N$, $u_{ki}u_{kj}^{*} = u_{ik}u_{jk}^{*} = 0$ as soon as $i\neq j$.
\end{itemize}
\end{de}

Recall that if $\G$ is a compact matrix quantum group, its \emph{free complexification} $\widetilde{\G}$ is defined in the following way : $C_{\text{max}}(\widetilde{\G})$ is the sub-C*-algebra of $C_{\text{max}}(\G)\ast C(S^{1})$ generated by the elements $u_{ij}z$, where $u$ is the fundamental representation of $\G$ and $z$ is the fundamental representation of $S^{1}$. If $\G = H_{N}^{+}$, the coefficients $v_{ij} = u_{ij}z$ satisfy the relations of the proposition above, giving a surjective $*$-homomorphism mapping $u_{ij}$ to $v_{ij}$. It is proven in \cite[Cor 2.5.13]{raum2012isomorphisms} that the fusion semiring of the free complexification of $H_{N}^{+}$ is the same as the one of $\widetilde{H}_{N}^{+}$. This, by virtue of \cite[Lem 5.3]{banica1999representations} implies that the above morphism is bijective. Hence we have proved :

\begin{cor}
Let $N\geqslant 4$ be an integer. Then, the easy quantum group associated to $\CC_{0^{+}}$ is the free complexification $\widetilde{H}_{N}^{+}$ of $H_{N}^{+}$.
\end{cor}

\begin{rem}
We could also prove directly that $\CC_{0^{+}}$ is the category of partitions of $\widetilde{H}_{N}^{+}$ and then use Theorem \ref{thm:classificationfree} to recover \cite[Cor 2.5.13]{raum2012isomorphisms}.
\end{rem}

\begin{rem}
The abelianization of $A_{h}^{0^{+}}(N)$ is the algebra $C(H_{N}^{\infty})$ of functions on the group $H_{N}^{\infty}$ of all unitary monomial matrices (i.e. having exactly one non-zero entry in each line and column) of size $N$. Moreover, the quotient of $A_{h}^{0^{+}}(N)$ by the relations $u = \overline{u}$ is the maximal C*-algebra $C_{\text{max}}(H_{N}^{+})$ of the free hyperoctahedral quantum group $H_{N}^{+}$, as expected.
\end{rem}

\section{One-dimensional representations}\label{sec:one}

In the general case (when there are one-dimensional representations), things become more complicated even if one still restricts to noncrossing partitions. One can however try to use the map $\Phi$, though it is ill-behaved with respect to the tensor product, and the nontrivial one-dimensional representations to study the quantum group $\G$. We will first study the possible one-dimensional representations which may appear for a free easy quantum group and then give some structure results for the group they form.

\subsection{Non-through-partitions}

One-dimensional representations of a compact quantum group form a group under the tensor product (the inverse being given by the contragredient), which will be denoted $\mathcal{G}(\G)$. We will of course study this group using partitions. Here is a basic but important fact.

\begin{lem}
Let $p$ be a projective partition. Then, $u_{p}$ is a one-dimensional representation if and only if $t(p) = 0$.
\end{lem}

\begin{proof}
The "if" part was proved in Lemma \ref{lem:trivial}. To prove the "only if" part, first note that if $u_{q}$ is equivalent to the trivial representation, then $t(q) = 0$. Let now $u_{p}$ be a one-dimensional representation and let $u_{\overline{p}}$ be its contragredient. Then, $u_{p}\otimes u_{\overline{p}}$ contains $u_{p\otimes \overline{p}}$, which must therefore be equivalent to the trivial representation. Hence, $t(p)\leqslant t(p\otimes \overline{p}) = 0$.
\end{proof}

We therefore only have to study partitions with no through-block. Let us write, for an integer $k\geqslant 0$, $\beta_{k} = \theta_{k}^{*}\theta_{k}$. This is a projective partition in $\NCc(k, k)$ consisting of $2$ blocks, an upper and a lower one, each having $k$ white points. By convention, $\beta_{0}$ is the empty partition and $\beta_{-k} = \overline{\beta}_{k}$. As for through-partitions, we can recover any projective partition with $t(p) = 0$ from the $\beta_{k}$'s up to equivalence. This is not completely obvious and will be the object of Lemma \ref{lem:1dwords}. We first need the following fact :

\begin{lem}\label{lem:1dwordsbis}
Let $\CC$ be a category of noncrossing partitions and let $B\in \CC(k, k)$ be a projective partition with $t(B) = 0$ and such that \emph{all the points are white}. If $B\neq\emptyset$, then there is an integer $1\leqslant l\leqslant k$ such that $\beta_{l}\in \CC$.
\end{lem}

\begin{proof}
The proof is by induction on $k$. If $k = 1$, then $B = \beta_{1}$. Assume that the result holds for all $n\leqslant k$. Up to considering tensor products, we may assume that $1$ and $k$ are in the same block and write $B = [w_{0}, b_{1}, \dots, b_{n-1}, w_{n}]\in \CC$ with $b_{i}^{*}b_{i}\in \CC$ for each $i$ by Lemma \ref{lem:cappingnonthroughblock}. If there is an index $i$ such that $b_{i}^{*}b_{i}\neq \emptyset$, we may apply our induction hypothesis to it to conclude. Otherwise, this precisely means that $B = \beta_{k}$.
\end{proof} 

\begin{lem}\label{lem:1dwords}
Let $\CC$ be a category of noncrossing partitions (which is not block-stable) and let $l$ be the smallest strictly positive integer such that $\beta_{l}\in \CC$. Then, any projective partition $p\in \CC(k, k)$ with $t(p) = 0$ is equivalent to $\beta_{l}^{\otimes m}$ or to $\beta_{-l}^{\otimes m}$ for some integer $m$ (with the convention that $\beta_{l}^{\otimes 0} = \beta_{0}$).
\end{lem}

\begin{proof}
Up to equivalence, we may assume by capping that all the points of $p$ have the same color (if $p$ has the same number of black and white points on each row, then $p_{u}\in \CC$ so that $p$ is equivalent to the empty partition and the result holds). If this color is black, we can consider $\overline{p}$ instead of $p$ to turn all the points into white. Then, Lemma \ref{lem:1dwordsbis} tells us that for $k\leqslant l$, we can only get the empty partition. For $k \geqslant l$, let $k = m\times l + r$ be the euclidian division of $k$ by $l$ and set
\begin{equation*}
x = p(\beta_{l}^{\otimes m}\otimes \pi_{1}^{\otimes r}) \text{ and } q = x^{*}x.
\end{equation*}
Then, $x$ implements an equivalence between $p$ and $q$ and $q = \beta_{l}^{\otimes m}\otimes y$, where $y\in\CC(r, r)$ is a projective partition with $t(y) = 0$. Since $r<l$, $y = \emptyset$, concluding the proof.
\end{proof}

Let us highlight a nontrivial consequence of this fact :

\begin{prop}
Let $N\geqslant 4$ be an integer, let $\CC$ be a category of \emph{noncrossing} partitions and let $\G$ be the associated easy quantum group. Then, the group $\mathcal{G}(\G)$ of one-dimensional representations of $\G$ is cyclic (and in particular abelian).
\end{prop}

\subsection{Classification}

We now want to classify the one-dimensional representations of $\G$. In view of Lemma \ref{lem:1dwords}, we shall focus on the set
\begin{equation*}
J(\CC) = \{k\in \Z, \beta_{k}\in \CC\}.
\end{equation*}
In fact, quotienting out $J(\CC)$ by the equivalence relation $k\sim k' \Leftrightarrow \beta_{k}\sim \beta_{k'}$ yields a group $\mathcal{G}(\CC)$ (for $\otimes$) isomorphic to $\mathcal{G}(\G)$. Note that rotating $\beta_{k}$ on one line and capping in the middle yields a rotated version of $\pi_{k-1}$, so that $k\in J(\CC) \Rightarrow k-1\in I(\CC)$. This observation will simplify the study of $J(\CC)$.

\begin{lem}
Let $\CC$ be a category of noncrossing partitions. Then, the following hold :
\begin{enumerate}
\item If $I(\CC) = \Z^{*}$, then there is an integer $n\geqslant 0$ such that $J(\CC) = n\Z$. If moreover $S(\CC) = \Z_{s}$, then $n$ divides $s$.
\item If $I(\CC) = \{-1, 1\}$, then $J(\CC) \subset \{-2, -1, 0, 1, 2\}$.
\end{enumerate}
\end{lem}

\begin{proof}
$(1)$ : Noticing that
\begin{equation*}
(\beta_{k}\otimes \beta_{k'})\pi_{k+k'} = r^{\beta_{k+k'}}_{\beta_{k}\otimes \beta_{k'}},
\end{equation*}
we see that if $I(\CC) = \Z^{*}$, then $J(\CC)$ is stable by addition (and $\beta_{k+k'} \sim \beta_{k}\otimes \beta_{k'}$). It is thus an additive subgroup of $\Z$ and is equal to $n\Z$ for some $n$. This $n$ is the smallest positive integer $k$ such that $\beta_{k}\in \CC$, so that in particular $n\leqslant s$ as soon as $\theta_{s}\in \CC$. Assume that $S(\CC) = \Z_{s}$, let $s = n\times m + r$ be the euclidian division of $s$ by $n$ and set $x = \theta_{s}(\beta_{n}^{\otimes m}\otimes \pi_{1}^{\otimes r})$. Then, $x^{*}x = \beta_{n}^{\otimes n}\otimes \beta_{r}$ and therefore $\beta_{r}\in \CC$. This implies that $r = 0$, hence $n$ divides $s$.

$(2)$ : This is clear from the fact that $k\in J(\CC) \Rightarrow k-1\in I(\CC)$.
\end{proof}

Note that in the second case, there are in fact three possibilities : $\{-2, -1, 0, 1, 2\}$, $\{-2, 0, 2\}$ and $\{-1, 0, 1\}$. Deriving the structure of $\mathcal{G}(\CC)$ from that of $S(\CC)$ is now straightforward.

\begin{thm}\label{thm:1d}
Let $\CC$ be a category of noncrossing partitions.
\begin{enumerate}
\item If $I(\CC) = \Z^{*}$ and $S(\CC) = \Z_{s}$, then $\mathcal{G}(\CC) = \Z_{d}$ with $d = s/n$ if $J(\CC) = n\Z$.
\item If $I(\CC) = \{-1, 1\}$ and $1\in J(\CC)$, then $\mathcal{G}(\CC) = \Z_{s}$ with $s = \min\{k\in \N, \beta_{1}^{\otimes k}\sim \emptyset\}$.
\item If $I(\CC) = \{-1, 1\}$ and $1\notin J(\CC)$, then $\mathcal{G}(\CC) = \Z_{s}$ with $s = \min\{k\in \N, \beta_{2}^{\otimes k}\sim \emptyset\}$.
\end{enumerate}
\end{thm}

\begin{proof}
$(1)$ : The basic remark is that rotating $r^{\beta_{k}}_{\beta_{k'}}$ and capping yields $\theta_{\vert k - k'\vert}$. Thus, $S(\CC) = \Z_{s}$ means that $\beta_{k}\sim \beta_{k'}$ if and only if $k = k' [s]$, which is precisely the statement.

$(2)$ : It is known by Lemma \ref{lem:1dwords} that any one-dimensional representation is equivalent to a tensor power of $\beta_{1}$, hence the result.

$(3)$ : This is exactly the same reasoning as above.
\end{proof}

Important for the sequel will be to know whether $\mathcal{G}(\CC)$ is finite or not. Here is what we can deduce from this section.

\begin{cor}\label{cor:finite1d}
Let $\CC$ be a category of noncrossing partitions. Then, $\mathcal{G}(\CC)$ is a finite group if and only if it is trivial or there are integers $s, k\geqslant 1$ such that $\theta_{s}^{\otimes k}\in \CC$.
\end{cor}

\begin{proof}
From what precedes, we see that if $\mathcal{G}(\CC)$ is nontrivial, then it is finite if and only if $\CC$ satisfies one of the following conditions :
\begin{itemize}
\item $\theta_{s}\in \CC$ for some integer $s\geqslant 3$.
\item $\theta_{1}^{\otimes k}\in \CC$ for some integer $s\geqslant 1$ (this is equivalent to $\beta_{1}^{\otimes k}\sim \emptyset$).
\item $\theta_{2}^{\otimes k}\in \CC$ for some integer $s\geqslant 1$ (this is equivalent to $\beta_{2}^{\otimes k}\sim \emptyset$).
\end{itemize}
Assume that $\theta_{s}^{\otimes k}\in \CC$ for some integer $k$. Then, $\pi_{s-1}\in \CC$ and if $ks = m\times (s-1) + r$ is the euclidian division of $ks$ by $s-1$, we have
\begin{equation*}
r^{\pi_{(s-1)\times (m+1)}}_{\pi_{s-1-r}} = (\theta_{s}^{\otimes k}\otimes \pi_{1}^{\otimes (s-1-r)})(\pi_{s-1}^{\otimes (m+1)}) \in \CC.
\end{equation*}
This means that $\theta_{\vert (s-1)\times (m+1) - (s-1-r)\vert} = \theta_{ks}\in \CC$, concluding the proof.
\end{proof}

\section{Applications}\label{sec:applications}

\subsection{Length functions}\label{subsec:length}

Length functions on discrete quantum groups were introduced in \cite[Def 3.1]{vergnioux2007property}. Any compact matrix quantum group is endowed with a natural "word length function" given, for an irreducible representation $\alpha$, by
\begin{equation*}
L(\alpha) = \inf\{k, \alpha\subset u^{\otimes w} \text{ with } \vert w\vert = k\}.
\end{equation*}
This length function is central and proper. However, the structure of free fusion ring gives another length function, inherited from the length function on the underlying free monoid, or equivalently from the through-block structure of the projective partitions.

\begin{de}\label{de:length}
Let $\G$ be an easy quantum group and let $\CC$ be its associated category of partitions. If $p\in \Proj$, we set $\ell(u_{p}) = t(p)$. This defines a central length function on $\G$.
\end{de}

Note that Definition \ref{de:length} makes sense for any easy quantum group but is ill-behaved in general. For instance, nontrivial one-dimensional representations have length $0$. This problem can easily be overcome by setting $\ell'(u_{p}) = \ell(u_{p}) + \delta_{p\sim\emptyset}$, but the crucial issue is rather whether this length function is proper or not. Let us characterize precisely when this is the case.

\begin{prop}\label{prop:properlength}
Let $\CC$ be a category of noncrossing partitions. Then, the length function $\ell$ is proper if and only if both $S(\CC)$ and $\mathcal{G}(\CC)$ are finite.
\end{prop}

\begin{proof}
Assume that $\vert S(\CC)\vert = s$. Then, according to Corollary \ref{cor:generators}, to build a projective partition $p$ with $t(p) = k$ we have to chose :
\begin{itemize}
\item $k$ elements $A_{1}, \dots A_{k}$ in $S(\CC)$ : $s^{k}$ choices.
\item Between $A_{i}$ and $A_{i+1}$, before $A_{1}$ and after $A_{k}$, an element of $\mathcal{G}(\CC)$ : $\vert\mathcal{G}(\CC)\vert^{k+1}$ choices.
\end{itemize}
Hence, we have
\begin{equation*}
\ell^{-1}(\{k\}) \leqslant s^{k}\vert\mathcal{G}(\CC)\vert^{k+1},
\end{equation*}
yielding the "if" part of the statement. Moreover, we obviously have
\begin{equation*}
\ell^{-1}(\{1\}) \geqslant s \text{ and } \ell^{-1}(\{0\}) = \vert \mathcal{G}(\CC)\vert,
\end{equation*}
giving the "only if" part of the statement.
\end{proof}

\begin{rem}
The previous reasoning can also be used to obtain a lower bound. In fact, $\ell^{-1}(\{k\})$ contains at least all words of length $k$ on $S(\CC)$ multiplied by an element of $\mathcal{G}(\CC)$, hence
\begin{equation*}
\ell^{-1}(\{k\}) \geqslant \vert \mathcal{G}(\CC)\vert s^{k}.
\end{equation*}
\end{rem}

\begin{rem}\label{rem:equivalencelengthfunctions}
Assume that $\vert S(\CC)\vert = s$, $\vert \mathcal{G}(\CC)\vert < +\infty$ and that any non-through-block projective partition is equivalent to a partition with at most $2D$ points. Then, $\ell$ and $L$ are equivalent in the following sense : for any $\alpha\in \Irr(\G)$,
\begin{equation*}
\ell(\alpha) \leqslant L(\alpha)\leqslant (s+D)L(\alpha) + D
\end{equation*}
This comes from the fact that, up to equivalence, a partition $p$ with $t(p)=k$ has at most $sk + (k+1)D = (s+D)k+D$ points.
\end{rem}

Noticing that if $S(\CC)$ is infinite, then $\mathcal{G}(\CC)$ is either infinite or trivial, we get the following corollary :

\begin{cor}
Let $\CC$ be a category of noncrossing partitions \emph{which is not $\CC_{\infty}$}. Then, $\ell$ is proper if and only if $\mathcal{G}(\CC)$ is finite.
\end{cor}

\begin{proof}
If $I(\CC) = \{-1, 1\}$, then $S(\CC)$ is finite and $\ell$ is proper as soon as $\mathcal{G}(\CC)$ is finite.

If $I(\CC) = \Z^{*}$ and $\mathcal{G}(\CC)$ is trivial, then it is clear that $\ell$ is proper as soon as $\CC\neq \CC_{\infty}$.

Assume eventually that $I(\CC) = \Z^{*}$ and that $\mathcal{G}(\CC)$ is finite and nontrivial. By Corollary \ref{cor:finite1d}, $\theta_{s}^{\otimes k}\in \CC$ for some integers $s$ and $k$, implying that $\theta_{ks}\in \CC$. This means that $S(\CC)$ is finite and thus $\ell$ is proper.
\end{proof}

\subsection{The Haagerup property}

We now turn to approximation properties for free easy quantum groups. More precisely, we will give a unified proof of the Haagerup property for free easy quantum groups such that $\ell$ is proper. This will be achieved using the Haagerup property for $S_{N}^{+}$ proved by M. Brannan in \cite{brannan2012reduced} and the properness of the length function $\ell$ studied in the previous subsection.

Let us first recall some facts concerning the Haagerup property. Because the quantum groups we are studying are of \emph{Kac type}, we can restrict our attention, as far as approximation properties are concerned, to characters of representations.

\begin{de}
Let $\G$ be a compact quantum group and let $v\in C_{\text{max}}(\G)\otimes \B(H)$ be a finite-dimensional representation of $\G$. Its \emph{character} is defined by
\begin{equation*}
\chi_{v} = (\ii\otimes \Tr)(v)\in C_{\text{max}}(\G).
\end{equation*}
\end{de}

It is proved in \cite[Cor 5.9]{woronowicz1987compact} that two representations are unitarily equivalent if and only if their characters are equal. Moreover, we have by \cite[Thm 5.8]{woronowicz1987compact} that
\begin{equation*}
\chi_{u\oplus v} = \chi_{u} + \chi_{v} \text{ and } \chi_{u\otimes v} = \chi_{u}\chi_{v}.
\end{equation*}
In other words, the (non-closed) algebra $\Pol(\G)_{0}$ generated in $C_{\text{max}}(\G)$ by the characters is isomorphic to the \emph{complexified fusion ring} $R(\G)\otimes_{\Z}\C$ of $\G$. The Haagerup property admits a simple description at the level of characters.

\begin{de}
A compact quantum group $\G$ \emph{of Kac type} is said to have the \emph{Haagerup property} if there is a net $(\varphi_{i})_{i}$ of states on the algebra of characters $\Pol(\G)_{0}$ such that
\begin{enumerate}
\item $(\varphi_{i})_{i}$ converges pointwise to the counit (equivalently, for any $\alpha\in \Irr(\G)$, $\varphi_{i}(\chi_{\alpha})\underset{i}{\rightarrow} \dim(\alpha)$).
\item For any $i$ and for any $\epsilon > 0$, there is a finite subset $F\subset \Irr(\G)$ such that for any $\alpha\notin F$,
\begin{equation*}
\left\vert \frac{\varphi_{i}(\chi_{\overline{\alpha}})}{\dim(\alpha)}\right\vert \leqslant \epsilon.
\end{equation*}
\end{enumerate}
\end{de}

\begin{rem}
Our definition of the Haagerup property looks a bit different from that of  \cite{brannan2011approximation} but both are shown to be equivalent (as well as several other characterizations) in \cite{daws2014haagerup}.
\end{rem}

Recall that if $\G$ is a free easy quantum group, then there is a canonical surjection
\begin{equation*}
\Pi : C_{\text{max}}(\G) \rightarrow C_{\text{max}}(S_{N}^{+})
\end{equation*}
characterized by the fact that it sends the fundamental representation of $\G$ onto the fundamental representation of $S_{N}^{+}$. Following the strategy of F. Lemeux in \cite{lemeux2013haagerup}, we will use $\Pi$ to pull back the states giving the Haagerup property on $S_{N}^{+}$. Let us denote by $(\varphi_{i})_{i}$ any net of states implementing the Haagerup property on $S_{N}^{+}$. The natural states to look at for a general free easy quantum group $\G$ are
\begin{equation*}
\psi_{i} = \varphi_{i}\circ \Pi.
\end{equation*}
For clarity, we will first deal with the computational part of the proof. For $p\in \Proj$, we write
\begin{equation*}
\Pi(\chi_{p}^{\G}) = \sum_{k=0}^{t(p)}A_{k}(p)\chi_{k}^{S_{N}^{+}}
\end{equation*}
where $A_{k}(p)$ is a positive integer.

\begin{lem}\label{lem:decomposition surjection}
Let $K_{0} > 0$ be an integer. If $N > 4(s+D)$, where $\vert S(\CC)\vert = s$ and any one-dimensional representation has a representative with at most $2D$ points, then
\begin{equation*}
\frac{\displaystyle\sum_{k=0}^{K_{0}}A_{k}(p)\varphi_{i}\left(\chi_{k}^{S_{N}^{+}}\right)}{\displaystyle\sum_{k=0}^{t(p)}A_{k}(p)\dim\left(u_{k}^{S_{N}^{+}}\right)} \longrightarrow 0
\end{equation*}
as $t(p)\to +\infty$.
\end{lem}

\begin{proof}
First note that all the terms appearing in the quotient are positive. Moreover, we can choose the net $(\varphi_{i})_{i\in [i_{0}, N]}$ such that for all $k$,
\begin{equation*}
\varphi_{i}\left(\chi_{k}^{S_{N}^{+}}\right)\leqslant C_{0}\left(\frac{i}{N}\right)^{k}\dim\left(u_{k}^{S_{N}^{+}}\right)\leqslant C_{0}\dim\left(u_{k}^{S_{N}^{+}}\right)
\end{equation*}
where $C_{0}$ is a constant depending only on $i_{0} > 4$ (see the proof of \cite[Thm 4.2]{brannan2012reduced}). As a consequence, it is enough to compute the limit of
\begin{equation*}
\frac{\displaystyle\sum_{k=0}^{K_{0}}A_{k}(p)\dim\left(u_{k}^{S_{N}^{+}}\right)}{\displaystyle\sum_{k=0}^{t(p)}A_{k}(p)\dim\left(u_{k}^{S_{N}^{+}}\right)}.
\end{equation*}
It is clear that $A_{t(p)}(p) = 1$, so that the denominator is greater than $\dim\left(u_{t(p)}^{S_{N}^{+}}\right)$ which is known to grow as $N^{t(p)}$. As for the numerator, it can be bounded by
\begin{equation*}
K_{0}\times\dim(u_{K_{0}}^{S_{N}^{+}})\times\max_{k\leqslant K_{0}}A_{k}(p).
\end{equation*}
Moreover, the number $A_{k}(p)$ is at least bounded by the total number of partitions on $L\left(u_{p}^{\G}\right)$ points, i.e. the Catalan number $C_{L\left(u_{p}^{\G}\right)}$. Gathering these estimates, the quantity that we are interested in is less than
\begin{equation*}
K_{0}\dim\left(u_{K_{0}}^{S_{N}^{+}}\right)\frac{C_{L\left(u_{p}^{\G}\right)}}{N^{t(p)}}.
\end{equation*}
To conclude, simply use the well-known estimate $C_{n} \sim \pi^{-1/2}n^{-3/2}4^{n}$ as well as the fact that $L(u_{p}^{\G})\leqslant (s+D)\ell(u_{p}^{\G}) + D = (s+D)t(p) + D$.
\end{proof}

\begin{rem}
For $O_{N}^{+}$ and $U_{N}^{+}$, we have $L(\alpha) = \ell(\alpha)$ so that the estimates work for all $N > 4$.
\end{rem}

We are now ready for the proof of the Haagerup property.

\begin{thm}\label{thm:generalhaagerup}
Let $\CC$ be a category of \emph{noncrossing} partitions such that \emph{the length function $\ell$ is proper}, let $N\geqslant 4^{s+D}$ be an integer and let $\G$ be the associated easy quantum group. Then, $\G$ has the Haagerup property.
\end{thm}

\begin{proof}
Let $(\varphi_{i})_{i}$ be a net of states implementing the Haagerup property for $S_{N}^{+}$. We claim that the associated net $(\psi_{i})_{i}$ does the job. The fact that it converges pointwise to the counit is clear because $(\varphi_{i})_{i}$ does and $\Pi$ is a Hopf $*$-algebra morphism. We therefore only have to check that for a fixed $i$,
\begin{equation*}
\frac{\psi_{i}\left(\chi_{\overline{p}^{\G}}\right)}{\dim\left(u_{p}^{\G}\right)} \longrightarrow 0
\end{equation*}
outside finite sets of equivalence classes of irreducible representations. Since $\ell$ is a proper length function on $\G$, it is equivalent to prove that the above quantity tends to $0$ as $\ell\left(u_{p}^{\G}\right)\rightarrow \infty$. Let $\epsilon >0$ and let $K_{0}$ be an integer such that, for all $k\geqslant K_{0}$,
\begin{equation*}
\frac{\varphi_{i}\left(\chi_{k}^{S_{N}^{+}}\right)}{\dim\left(u_{k}^{S_{N}^{+}}\right)} \leqslant \epsilon/2,
\end{equation*}
Then,
\begin{equation*}
\frac{\psi_{i}\left(\chi_{p}^{\G}\right)}{\dim\left(u_{p}^{\G}\right)} = \frac{1}{\dim\left(u_{p}^{\G}\right)}\sum_{k=0}^{K_{0}}A_{k}(p)\varphi_{i}\left(\chi_{k}^{S_{N}^{+}}\right) + \sum_{k > K_{0}}\frac{A_{k}(p)\varphi_{i}\left(\chi_{t(q)}^{S_{N}^{+}}\right)}{\dim\left(u_{p}^{\G}\right)}.
\end{equation*}
Because $\Pi$ is a Hopf $*$-algebra homomorphism it is dimension-preserving, hence $\dim\left(u_{p}^{\G}\right) = \sum A_{k}(p)\dim\left(u_{k}^{S_{N}^{+}}\right)$. Thus, by Lemma \ref{lem:decomposition surjection}, there is an integer $K_{1}$ such that the first term is less than $\epsilon/2$ as soon as $\ell(p) = t(p) \geqslant K_{1}$. The second term can be bounded by
\begin{equation*}
\frac{\epsilon}{2\dim\left(u_{p}^{\G}\right)}\sum_{k> K_{0}}A_{k}(p)\dim\left(u_{t(q)}^{S_{N}^{+}}\right) \leqslant \frac{\epsilon}{2}.
\end{equation*}
Combining the two estimates, we have, for $\ell\left(u_{p}^{\G}\right)\geqslant \max(K_{0}, K_{1})$,
\begin{equation*}
\frac{\psi_{i}\left(\chi_{\overline{p}}^{\G}\right)}{\dim\left(u_{p}^{\G}\right)} \leqslant \epsilon
\end{equation*}
and the result follows.
\end{proof}

\begin{rem}
For $N\geqslant 5$, the quantum group $S_{N}^{+}$ is not amenable by \cite{banica1999symmetries}. Since amenability passes to quantum subgroups, we can infer that a free easy quantum group is never amenable when $N\geqslant 5$.
\end{rem}

Theorem \ref{thm:generalhaagerup} applies in particular to free quantum groups without nontrivial one-dimensional representations. In that case, we know by Proposition \ref{prop:properlength} that $\ell$ is proper provided the quantum group is not $H_{N}^{\infty +}$. Hence the following corollary :

\begin{cor}
The following quantum groups have the Haagerup property for $N$ large enough : $O_{N}^{+}$, $U_{N}^{+}$, $\widetilde{H}_{N}^{+}$ and $H_{N}^{s+}$ for $\infty > s \geqslant 2$.
\end{cor}

This recovers previous results of M. Brannan \cite{brannan2011approximation} and F. Lemeux \cite{lemeux2013haagerup}. In particular, Theorem \ref{thm:generalhaagerup} gives explicit multipliers implementing the Haagerup property on $U_{N}^{+} = \widetilde{O}_{N}^{+}$ and $\widetilde{H}_{N}^{+}$ without resorting to a free product trick.

\subsection{Recovering the fusion ring}

In this section we adress the question of reconstructing the fusion ring $R^{+}(\G)$ from $\RR$ and $\mathcal{G}(\CC)$. Using the map $\Phi$, we get an additive subsemigroup $R^{+}(\CC) = \Phi(\RR)$ of $R^{+}(\G)$ (even though it is not a subsemiring in general). Similarly, we can see $\N[\mathcal{G}(\CC)]$ as a subsemiring of $R^{+}(\G)$. This data is enough to recover the fusion \emph{ring} of $\G$.

\begin{prop}\label{prop:generatingthering}
Let $N\geqslant 4$, let $\CC$ be a category of noncrossing partitions and let $\G$ be the associated easy quantum group. Then, $R(\G)$ is generated as a \emph{ring} by $R^{+}(\CC)$ and $\mathcal{G}(\CC)$.
\end{prop}

\begin{proof}
Let $R$ denote the subring of $R(\G)$ generated by $R^{+}(\CC)$ and $\mathcal{G}(\CC)$ and let us prove by induction on $t(p)$ that $[u_{p}]$ is in $R$. If $t(p) = 0$ , then $[u_{p}]\in \mathcal{G}(\CC)\subset R$.

Assume now that $t(p) > 0$ and let $B_{0}\otimes A_{1}\otimes \dots\otimes A_{t(p)}\otimes B_{t(p)}$ be an equivalent projective partition given by Corollary \ref{cor:generators}. Then, $u_{p}$ is equivalent to a subrepresentation of $v = u_{B_{0}}\otimes u_{A_{1}}\otimes \dots\otimes u_{A_{t(p)}}\otimes u_{B_{t(p)}}$, which is in $R$ by definition. Moreover, all the other subrepresentations of $v$ are associated to partitions $q$ with $t(q) < t(p)$. Therefore, they are in $R$ by the induction hypothesis. We conclude that $u_{p}\in R$.
\end{proof}

This proposition does not give an explicit description of the fusion ring. Such a description is probably quite complicated in general, and we will focus on a particular case : when $R^{+}(\CC)$ is as a subsemiring of $R^{+}(\G)$.

\begin{prop}
Let $N\geqslant 4$ be an integer, let $\CC$ be a category of noncrossing partitions and let $\G$ be the associated easy quantum group. If $R^{+}(\CC)$ is a subsemiring of $R^{+}(\G)$, then exactly one of the following holds :
\begin{enumerate}
\item $\mathcal{G}(\CC)$ is trivial.
\item $\G\in \{B_{N}^{+}\times \Z_{2}, S_{N}^{+}\times \Z_{2}, B_{N}^{+}\ast \Z_{2}\}$.
\item $S(\CC) = \{[-1], [0^{+}], [0^{-}], [1]\}$ and $J(\CC) = \{-1, 0, 1\}$.
\item $S(\CC) = \{[-1], [1]\}$ and $J(\CC) = \{-1, 0, 1\}$.
\end{enumerate}
\end{prop}

\begin{proof}
First note that $\beta_{k}$ appears as a subpartition of $\pi_{1}^{\otimes k}$ for any $k\geqslant 2$. This means that we have the following alternative : either $\beta_{k}\notin \CC$ or $\beta_{k}\sim \emptyset$, which means that $\theta_{k}\in \CC$. In particular, $\mathcal{G}(\CC)$ is nontrivial only if $\beta_{1}\in \CC$.

Assume therefore that $\beta_{1}\in \CC$. If $\beta_{2}\in \CC$, then $\theta_{2}\in \CC$ and $\G\subset O_{N}^{+}$. Otherwise, we must have $J(\CC) = \{-1, 0, 1\}$, hence $I(\CC) = \{-1, 1\}$. If $1\sim -1$, then again $\G\subset O_{N}^{+}$. We are then in the case $(1)$ or $(2)$ (using for case $(2)$ the classification of all free easy orthogonal quantum groups given in \cite[Thm 2.9]{weber2012classification}).

Summarizing, we have proven that excluding $(1)$ and $(2)$, we must have $I(\CC) = \{-1, 1\}$, $-1\nsim 1$, $\beta_{1}\in \CC$ and $\beta_{2}\notin \CC$. This gives us either $S(\CC) = \{[-1], [1]\}$ or $S(\CC) = \mathcal{S}$.
\end{proof}

This result looks rather incomplete since the last two cases are not explicitely described. Such a description, however, will appear quite straightforwardly as a consequence of the classification of all free unitary easy quantum groups in \cite{tarrago2015unitary}.

\bibliographystyle{amsplain}
\bibliography{../../quantum}

\end{document}